\documentclass[a4paper]{amsart}

\usepackage{amssymb, enumitem}
\usepackage{hyperref, aliascnt}

\usepackage{xcolor}

\newcommand{\andSep}{\,\,\,\text{ and }\,\,\,}
\newcommand{\axiomO}[1]{(O#1)}

\newcommand{\NN}{{\mathbb{N}}}

\newcommand{\ca}{$C^*$-algebra}

\newcommand{\llideal}{\ensuremath{<\hskip-7pt\lhd\hskip3pt{}}}

\DeclareMathOperator{\Cu}{Cu}

\def\today{\number\day\space\ifcase\month\or   January\or February\or
   March\or April\or May\or June\or   July\or August\or September\or
   October\or November\or December\fi\   \number\year}

\newtheorem{lma}{Lemma}[section]

\newaliascnt{thmCt}{lma}
\newtheorem{thm}[thmCt]{Theorem}
\aliascntresetthe{thmCt}

\newaliascnt{corCt}{lma}
\newtheorem{cor}[corCt]{Corollary}
\aliascntresetthe{corCt}

\newaliascnt{prpCt}{lma}
\newtheorem{prp}[prpCt]{Proposition}
\aliascntresetthe{prpCt}

\theoremstyle{definition}

\newaliascnt{dfnCt}{lma}
\newtheorem{dfn}[dfnCt]{Definition}
\aliascntresetthe{dfnCt}

\newaliascnt{rmkCt}{lma}
\newtheorem{rmk}[rmkCt]{Remark}
\aliascntresetthe{rmkCt}

\newaliascnt{rmksCt}{lma}

\aliascntresetthe{rmksCt}

\newaliascnt{exaCt}{lma}

\aliascntresetthe{exaCt}

\newaliascnt{conjCt}{lma}
\newtheorem{conj}[conjCt]{Conjecture}
\aliascntresetthe{conjCt}

\newaliascnt{exasCt}{lma}
\newtheorem{exas}[exasCt]{Examples}
\aliascntresetthe{exasCt}

\newaliascnt{qstCt}{lma}
\newtheorem{qst}[qstCt]{Question}
\aliascntresetthe{qstCt}

\newaliascnt{pgrCt}{lma}
\newtheorem{pgr}[pgrCt]{}
\aliascntresetthe{pgrCt}

\newcounter{theoremintro}

\newaliascnt{thmIntroCt}{theoremintro}

\aliascntresetthe{thmIntroCt}

\newaliascnt{dfnIntroCt}{theoremintro}

\aliascntresetthe{dfnIntroCt}

\newaliascnt{prpIntroCt}{theoremintro}

\aliascntresetthe{prpIntroCt}

\newaliascnt{corIntroCt}{theoremintro}

\aliascntresetthe{corIntroCt}

\newaliascnt{qstIntroCt}{theoremintro}
\newtheorem{qstIntro}[qstIntroCt]{Question}
\aliascntresetthe{qstIntroCt}

\newcommand{\K}{\mathcal{K}}

\numberwithin{equation}{section}

\newcommand{\nocontentsline}[3]{}
\let\origcontentsline\addcontentsline
\newcommand\stoptoc{\let\addcontentsline\nocontentsline}
\newcommand\resumetoc{\let\addcontentsline\origcontentsline}

\title{An introduction to the Global Glimm Problem}

\author{Eduard Vilalta}
\address{Eduard~Vilalta, 
Department de Matem\`{a}tiques, Universitat Polit\`{e}cnica de Catalunya, Diagonal 647, Barcelona.}
\email{eduard.vilalta@upc.edu}
\urladdr{www.eduardvilalta.com}

\thanks{
The author was partially supported by the Knut and Alice Wallenberg Foundation (KAW 2021.0140) and the Spanish State Research Agency (grant No. PID2023-147110NB-I00).
}

\subjclass[2020]%
{Primary
46L05. 
}

\begin{document}

\begin{abstract}
The Global Glimm Problem lies at the heart of several open questions regarding regularity properties of \ca{s}. The problem has been open for over two decades, and has recently garnered significant attention due to its strong ties to non-simple versions of well-known dimension reduction phenomena for simple \ca{s}, such as the weakly purely infinite problem and the non-simple Toms-Winter conjecture.

This survey offers a concise introduction to the problem, detailing some of the key concepts and methods used in the partial solutions achieved thus far. Particular emphasis is placed on the implications that a complete resolution of the Global Glimm Problem would have for other open questions in the field. 
\end{abstract}

\maketitle

\tableofcontents

\section{Introduction}

A recurring theme in the modern theory of \ca{s} is the prevalence of dimension reduction phenomena. In various contexts, algebras that exhibit a sufficient degree of `non-commutativity' tend to behave as though they are low-dimensional, often displaying remarkable rigidity. This phenomenon manifests as a form of robustness, where apparently weak versions of regularity properties automatically upgrade to their strongest forms.

Among the many open problems in this area, this survey focuses on two specific questions that will serve to motivate our examples and results throughout:
\begin{itemize}
    \item[(Q1)] Are all weakly purely infinite \ca{s} purely infinite?
    \item[(Q2)] Which unital separable \ca{s} of finite nuclear dimension are $\mathcal{Z}$-stable? For which separable nuclear \ca{s} does pureness imply (and therefore agrees with) $\mathcal{Z}$-stability?
\end{itemize}

While formal definitions of the previous concepts are deferred to later sections, a brief remark about (Q1) and (Q2) is in order. In the simple setting, there is a clear divide: neither weakly purely infinite, pure, nor $\mathcal{Z}$-stable \ca{s} are elementary (that is, they are not isomorphic to the compact operators over any Hilbert space). Consequently, the simple versions of (Q1) and (Q2) presuppose non-elementariness. Once this is assumed, both questions famously admit affirmative answers: simple weakly purely infinite \ca{s} are purely infinite \cite{KirRor02InfNonSimpleCalgAbsOInfty}, and unital separable simple non-elementary \ca{s} of finite nuclear dimension are $\mathcal{Z}$-stable\footnote{Although this statement does not look like a dimension reduction theorem, we will see in \autoref{rmk:TWisDR} that ---once suitably interpreted--- it is.} \cite{Win12NuclDimZstable}. This reflects a broader principle for simple \ca{s}: once one departs from the elementary case, dimension reduction phenomena become ubiquitous. Such an idea has been instrumental in some of the latest breakthroughs in the structure and regularity theory of such \ca{s}, especially in the Elliott classification program \cite{Gon02,GonLinNiu20,Win14,TikWhiWin17QDNuclear} and in the study of the Toms-Winter conjecture \cite{CasEviTikWhiWin21NucDimSimple,Ror04StableRealRankZ,Win12NuclDimZstable}.

The success of dimension reduction phenomena for simple \ca{s} naturally prompts the question: what is the correct analogue of non-elementariness for non-simple \ca{s}? Furthermore, does such an analogue suffice to trigger similar reduction theorems?

A compelling candidate for this analogue is the requirement that no nonzero ideal of a quotient is elementary. Indeed, the condition is necessary for affirmative answers to questions like (Q1) and (Q2), as properties such as pure infiniteness, pureness, and $\mathcal{Z}$-stability pass to both ideals and quotients. This notion is termed \emph{nowhere scatteredness}, and is discussed at length in \autoref{sec:NSCa}.

For separable \ca{s}, nowhere scatteredness admits a topological characterization: A \ca{} is nowhere scattered if and only if every closed subset of its spectrum is perfect (contains no isolated points in its induced topology). Thus, the spectrum of a nowhere scattered \ca{} is highly non-Hausdorff, which should be viewed as a topological enforcement of sufficient non-elementariness. Conjecturally, it is this `gluing together' of points in the spectrum that generates the robustness required for dimension-reduction phenomena to occur.

However, there are also other candidates. For example, a simple \ca{} is non-elementary if and only if each hereditary sub-\ca{} contains a full square-zero element, a fact that can be found in the proof of Glimm's \cite[Lemma~4]{Gli61Type1}. Extending this to the general setting, one might consider \ca{s} where every hereditary sub-\ca{} contains an `approximately full' square-zero element. Kirchberg and R\o{}rdam introduced this concept in their investigation of pure infiniteness \cite{KirRor02InfNonSimpleCalgAbsOInfty}, aptly naming it the \emph{Global Glimm Property}. Crucially, they established that (Q1) holds if and only if weakly purely infinite algebras possess this property (\autoref{prp:GGPandPIP}), thereby initiating its systematic study. They also noted that the Global Glimm Property implies nowhere scatteredness (\autoref{prp:GGPimpNSCa}).

From a philosophical standpoint, nowhere scatteredness and the Global Glimm Property represent two extremes, with any other reasonable definition of `sufficient non-elementariness' lying in between. As we will discuss, nowhere scatteredness should be thought of as a natural, well-behaved baseline necessary for reduction phenomena. In contrast, the Global Glimm Property is technically more intricate and harder to verify, yet significantly more powerful in proofs. Throughout, we will encounter examples where nowhere scatteredness is well understood but the Global Glimm Property is not, alongside theorems that rely on the Global Glimm Property but remain elusive under the (formally) weaker assumption of nowhere scatteredness.

At the core of these discussions lies a singular issue: Are there multiple distinct analogues of non-elementariness? In light of the preceding paragraph, the relevant question is:

\begin{qstIntro}[The Global Glimm Problem]\label{qstIntro:GGP}
    Are nowhere scatteredness and the Global Glimm Property equivalent?
\end{qstIntro}

The problem has been open for over two decades (cf. \cite[Question~2]{EllRor06Perturb}) and has recently attracted renewed attention due to a surge of interest in non-simple or non-nuclear \ca{s} satisfying specific regularity conditions --- most notably strict comparison \cite{AGKEP24:StCompTwisted,Rob25Self,Oza25Arx} and pureness \cite{AntPerThiVil24arX:PureCAlgs}. As we will discuss, a complete positive answer to \autoref{qstIntro:GGP} would lead to a positive solution to both (Q1) and (Q2), and would deepen our understanding on the structure and regularity properties of \ca{s}.

\stoptoc

\subsection*{Outline of the survey} The organization of this manuscript may appear unconventional at first glance. Experts in the field will quickly notice that several results introduced early on rely on deep techniques that are only developed later in \autoref{sec:CuntzSgp}. This ordering, however, is intentional: we believe it offers a gentler introduction to the key concepts and helps build intuition before delving into deeper technicalities. The first two sections are devoted to presenting the central notions underlying the Global Glimm Problem. In \autoref{sec:ResAndAppl}, we explore the significance of the problem and survey the known partial solutions. Finally, \autoref{sec:CuntzSgp} presents a unified framework for approaching the problem via the Cuntz semigroup.

Throughout, the focus is on illustrative proofs. We avoid the most technically demanding arguments in favor of accessibility, assuming only a standard background in operator algebras (e.g. \cite{Bla06OpAlgs}). However, we take the opportunity to present proofs ---or, at least, sketches--- for results that, while known to experts, do not appear explicitly in the literature. We also provide translated variants of certain proofs, adapting abstract Cuntz semigroup arguments into the more familiar language of $C^*$-algebras.

\subsection*{Acknowledgments} This survey started as my notes for a talk about the Global Glimm Problem that
I gave during the \emph{XXIV School of Mathematics Lluis Santal\'{o} 2025 - Structure and Approximation of \ca{s}}, held at the Palacio de la Magdalena in Santander (Spain) in July 2025. I would like to thank F.~Lled\'{o}, F.~Perera and H.~Thiel for organizing it and inviting me.

\resumetoc

\section{Nowhere scatteredness}\label{sec:NSCa}

As discussed in the introduction, dimension reduction phenomena are naturally inscribed in the study of \ca{s} that have no nonzero elementary ideals of quotients (henceforth called \emph{ideal-quotients}). This minimal condition is what we call nowhere scatteredness:

\begin{dfn}[cf. {\cite[Theorem~B]{ThiVil24NowhereScattered}}]\label{dfn:NowScat}
    A \ca{} is \emph{nowhere scattered} if it has no nonzero elementary ideal-quotients.
\end{dfn}

\begin{exas}\label{exas:NowScat} $ $
\begin{enumerate} 
    \item Simple non-elementary \ca{s} are nowhere scattered, since the only nonzero ideal-quotient of a simple \ca{} is the algebra itself.
    
    \item We say that a \ca{} is \emph{$\mathcal{Z}$-stable} if it absorbs the \emph{Jiang-Su algebra} $\mathcal{Z}$ tensorially, that is, $A\otimes\mathcal{Z}\cong A$. The algebra $\mathcal{Z}$ was originally introduced in \cite{JiaSu99Projectionless} as an infinite-dimensional analogue of the complex numbers, and has played a crucial role in the structure and classification theory of nuclear simple \ca{s}; see \cite{Whi23ICM,Win18ICM} for an overview.
    
    $\mathcal{Z}$-stable algebras are all nowhere scattered, since $\mathcal{Z}$-stability passes to ideals and quotients \cite[Corollary~3.3]{TomWin07ssa} and no nonzero elementary algebra is $\mathcal{Z}$-stable.

    \item A von Neumann algebra is nowhere scattered if and only if it has no type I part; see \cite[Proposition~3.4]{ThiVil24NowhereScattered}.

    \item Let $\Theta$ be a real skew-symmetric $2d\times 2d$ matrix, and let $A_\Theta$ be the associated \emph{noncommutative torus}. Recall that, by definition, $A_\Theta$ is the universal \ca{} generated by unitaries $u_1,\ldots ,u_{2d}$ subject to the relations
    \[
        u_j u_k = e^{2\pi i\Theta_{j,k}}u_k u_j.
    \]

    One can show that $A_\Theta$ is nowhere scattered if and only if $\Theta$ is non rational, that is, if $\Theta$ has at least one irrational entry. This is in turn equivalent to $A_\Theta$ being $\mathcal{Z}$-stable, a result that follows as a combination of \cite[Theorem~1.5]{BlaKumRor92ApproxCentralMatUnit} and \cite[Theorem~2.3]{TomWin08ASH}. As we will see in \autoref{sec:ResAndAppl}, this equivalence between $\mathcal{Z}$-stability and nowhere scatteredness holds in much greater generality.
\end{enumerate}
\end{exas}

The following list of permanence properties is not exhaustive, but covers most of the important constructions. For example, it follows from (1) that a \ca{} is nowhere scattered if and only if its stabilization is. A proof of (1)-(3) is given in \cite[Section~4]{ThiVil24NowhereScattered}, while (4)-(5) can be found in \cite[Section~2]{EnsVil25arX:Twisted}.

\begin{prp}\label{prp:PermProp} The following permanence properties hold:
\begin{enumerate}
    \item Morita equivalence: Let $A,B$ be Morita equivalent. Then, $A$ is nowhere scattered if and only if $B$ is.
    \item Extensions: Nowhere scatteredness passes to hereditary sub-\ca{s} and quotients. Further, if $I$ is an ideal of a \ca{} $A$, then $A$ is nowhere scattered if and only if $I$ and $A/I$ are.
    \item Inductive limits: An inductive limit of nowhere scattered \ca{s} is nowhere scattered.
    \item Tensor products: Let $A$ be exact, and let $A$ or $B$ be nowhere scattered. Then, $A\otimes_{\rm min} B$ is nowhere scattered.
    \item $C_0(X)$-algebras: A $C_0(X)$-algebra is nowhere scattered if and only if each of its fibers is nowhere scattered.
\end{enumerate}
\end{prp}

\begin{rmk}\label{rmk:CX_NSCvsGGP}
    As discussed in \autoref{prp:PermPropGGP}, the Global Glimm Property is also known to satisfy conditions (1)–(3). However, the proofs of these statements are considerably more involved --- so much so that the most efficient proof relies on the Cuntz semigroup characterization presented in \autoref{sec:CuntzSgp}.
    
    Notably, it is not known whether the Global Glimm Property behaves well in the context of bundles. This marks the first instance in these notes of the aforementioned tension between nowhere scatteredness and the Global Glimm Property: while \autoref{prp:PermProp}~(5) follows directly from \autoref{dfn:NowScat}, just establishing that $C_0(X)$-algebras with simple, non-elementary fibers and finite dimensional $X$ satisfy the Global Glimm Property requires deeper techniques \cite[Theorem~4.3]{BlaKir04GlimmHalving}.
\end{rmk}

\autoref{thm:MultChar} below presents a number of equivalent formulations of nowhere scatteredness. While we omit the proof of this result, the most significant characterizations are highlighted in the subsections that follow (\ref{subsec:TopChar} and \ref{subsec:TracChar}), each discussed individually.

\begin{thm}[cf. {\cite[Theorem~B]{ThiVil24NowhereScattered}}]\label{thm:MultChar}
    Let $A$ be a \ca{}. Then, the following are equivalent:
    \begin{itemize}
        \item[(i)] $A$ is nowhere scattered;
        \item[(ii)] every quotient of $A$ is antiliminal, that is, no quotient $A/I$ contains a nonzero positive element $a$ such that $\overline{a(A/I)a}$ is commutative;
        \item[(iii)] no quotient $A/I$ of $A$ contains a \emph{minimal open projection}, that is, no quotient $A/I$ contains a nonzero projection $p$ such that $p(A/I)p=\mathbb{C}p$;
        \item[(iv)] no hereditary sub-\ca{} of $A$ admits a one-dimensional irreducible representation;
        \item[(v)] no hereditary sub-\ca{} of $A$ admits a finite-dimensional irreducible representation;
        \item[(vi)] given any positive functional $\varphi$ on $A$ and any hereditary sub-\ca{} $B$, there exists a Haar unitary in $\tilde{B}$ for $\varphi$.
    \end{itemize}
\end{thm}

\begin{rmk}
    The concept of nowhere scatteredness had appeared implicitly in the literature long before \cite{ThiVil24NowhereScattered}. For example, it can be found in Kirchberg's and R\o{}rdam's study of pure infiniteness \cite[Lemma~4.14]{KirRor02InfNonSimpleCalgAbsOInfty} (in the form of \autoref{thm:MultChar}~(v)), as well as in Tikuisis' and Winter's introduction of \cite{TikWin14DecompRank}.
    
    Blanchard and Kirchberg were the first to formally name the property, calling it \emph{strict anti-liminality} in \cite[Remark~3.2]{BlaKir04PureInf} (phrasing it as \autoref{thm:MultChar}~(ii)). The name nowhere scatteredness appeared in \cite{Thi20arX:diffuseHaar}, where a positive functional $\varphi$ is said to be \emph{nowhere scattered} if it satisfies \autoref{thm:MultChar}~(vi)\footnote{Although this is not how nowhere scatteredness is formally defined for functionals, it is equivalent by \cite[Theorem~4.11]{Thi20arX:diffuseHaar}.}. In other words, a \ca{} is nowhere scattered if and only if all of its positive functionals are.
    
    A systematic study of the notion was not conducted until \cite{ThiVil24NowhereScattered}, where the name \emph{nowhere scatteredness} was kept for \ca{s} not only because of its relation to nowhere scattered positive functionals, but also to emphasize the stark contrast that it has with the already established concept of \emph{scatteredness} in both topology and \ca{s}; see \autoref{subsec:TopChar} below.
\end{rmk}

\subsection{\texorpdfstring{Topological spaces and scattered \ca{s}}{Topological spaces and scattered C*-algebras}}\label{subsec:TopChar}
A topological space is said to be \emph{perfect} if it contains no isolated points. Scatteredness is a classical notion in topology, which is at the extreme opposite of being perfect. The definition is as follows: 

\begin{dfn}
    A topological space $X$ is \emph{scattered} if every nonempty closed subset of $X$ contains an isolated point in the induced topology of the subset.
\end{dfn}

Scattered spaces have been studied extensively in topology as generalizations of discrete spaces. The simplest example of a space which is scattered but not discrete is the Sierpiński space, that is, the space consisting of two points $\{x,y\}$ with open sets $\emptyset$, $\{ x\}$ and $\{x,y\}$.

\begin{rmk}
    Let $X$ be a compact Hausdorff space. Then, $X$ is scattered if and only if every nonzero quotient of $C(X)$ contains a minimal open projection. Indeed, assume first that $X$ is scattered. Then, since any nonzero quotient of $C(X)$ is of the form $C(K)$ for some nonempty closed subset $K\subseteq X$, one can set the minimal open projection $p$ to be the indicator function of the isolated point in $K$. Conversely, a minimal open projection on $C(Y)$ for a compact Hausdorff space $Y$ is always an indicator function of an isolated point. This gives the desired equivalence.

    In general, a \ca{} is said to be \emph{scattered} if every nonzero quotient contains a minimal open projection. These algebras were introduced in \cite[Definition~2.1]{Jen77ScatteredCAlg} as those algebras whose each state is atomic. The equivalence with the definition given here is from \cite[Theorem~1.4]{GhaKos18NCCantorBendixson}.
\end{rmk}

As seen in \autoref{thm:MultChar}~(iii), nowhere scattered \ca{s} should be thought of as the complete opposite of scattered algebras. This viewpoint can also be observed from a purely topological perspective:

\begin{dfn}
    A topological space $X$ is \emph{nowhere scattered} if every closed subset of $X$ is perfect, that is, if no closed subset of $X$ contains an isolated point in the induced topology of the subset.
\end{dfn}

By definition, no point in a nowhere scattered space is closed. Thus, the examples of such spaces are highly non-Hausdorff.

\begin{exas} $ $
    \begin{enumerate}
        \item Consider the space $X=\mathbb{R}$ with the usual topology and $Y=\mathbb{R}$ with the trivial topology. Then, the space $X\times Y$ with the product topology (that is, the topology where the open subsets are bands) is nowhere scattered. One can also show that it is not a $T_0$-space.
        \item Let $X=[-1,1]$ and equip it with the \emph{overlapping interval topology}, that is, the topology with open basis given by the intervals $\{ (x,1]\mid x<0\}\cup \{ [-1,y)\mid y>0\}$. Then, $X$ is both $T_0$ and nowhere scattered.
    \end{enumerate}
\end{exas}

\begin{thm}[{\cite[Theorem~5.3]{ThiVil24NowhereScattered}}]\label{thm:TopCharNSCa}
    Let $A$ be a separable \ca{}. Then, $A$ is nowhere scattered if and only if $\widehat{A}$ is.
\end{thm}
\begin{proof}
    It is readily checked that a topological space is nowhere scattered if and only if it has no nonempty, scattered, locally closed subsets. Further, recall that locally closed subsets of $\widehat{A}$ correspond to ideal-quotients of $A$. Paired with the fact that a separable \ca{} is scattered if and only if $\widehat{A}$ is scattered~\cite[Corollary~3]{Jen78ScatteredCAlg2}, the result follows from the following Claim.

    \textbf{Claim.} \emph{A \ca{} is nowhere scattered if and only if none of its nonzero ideal-quotients is scattered.}

    To prove the Claim, assume first that $A$ is nowhere scattered. Then, it follows from \autoref{prp:PermProp} that all ideal-quotients are nowhere scattered as well. Thus, they are not scattered unless they are zero.

    For the converse, assume for the sake of contradiction that $A$ is not nowhere scattered. By definition, $A$ has a nonzero elementary ideal-quotient. However, elementary \ca{s} are scattered, which contradicts our assumption.
\end{proof}

\subsection{A `tracial' characterization}\label{subsec:TracChar}
We now present the characterization of nowhere scatteredness that has proven most useful when verifying the property. This can be viewed as a `tracial' characterization, in so far that nowhere scatteredness can be detected by examining the behavior of dimension functions. To set the stage, we begin by briefly recalling the concepts of Cuntz subequivalence and quasitraces.

Given positive elements $a,b$ in a \ca{} $A$, we write $a\precsim b$ whenever $a=\lim_n r_nbr_n^*$ for some sequence $(r_n)_n\subseteq A$. We also write $a\sim b$ whenever $a\precsim b$ and $b\precsim a$. These relations were introduced by Cuntz in \cite{Cun78DimFct}, and are nowadays known as \emph{Cuntz subequivalence} and \emph{equivalence} respectively.

Throughout the text we will use the following properties of $\precsim$. The reader is referred to \cite{GarPer23arX:ModernCu} for a proof of these statements\footnote{Items (1) and (2) of \autoref{prp:CuSubProp} are often included as part of what is known as `R\o{}rdam's lemma'.}.

\begin{lma}\label{prp:CuSubProp}
    Let $A$ be a \ca{} and let $a,b\in A_+$. Then,
    \begin{enumerate}
        \item $a\precsim b$ whenever $(a-\varepsilon)_+\precsim b$ for every $\varepsilon>0$.
        \item if $a\precsim b$, then for every $\varepsilon>0$ there exists $\delta >0$ and $r\in A$ such that $(a-\varepsilon)_+=r(b-\delta)_+r^*$.
        \item $a\precsim b$ whenever $a\in\overline{bAb}$.
        \item $(a-\varepsilon)_+\precsim b$ whenever $\Vert a-b\Vert <\varepsilon$.
        \item ${\rm diag}(a+b,0)\precsim {\rm diag}(a,b)$ in $M_2(A)$ always, and ${\rm diag}(a,b)\precsim {\rm diag}(a+b,0)$ whenever $a\perp b$.
    \end{enumerate}
\end{lma}

\begin{rmk}\label{rmk:NtnDiagSum}
    Following standard notation, we will henceforth use $\oplus$ to denote diagonal addition of elements in a \ca{}. For instance, the first subequivalence in (5) above may be expressed as $(a+b)\oplus0\precsim a\oplus b$. It is also customary to omit diagonal addition by 0, allowing us to write this subequivalence simply as $a+b\precsim a\oplus b$, with the  understanding that elements of $A_+$ are embedded in $M_2(A)_+$ via the map $a\mapsto a\oplus 0$.
\end{rmk}

Let $\tau\colon A_+\to [0,\infty]$ be a lower semicontinuous $2$-quasitrace on $A$ \cite[Definition~II.1.1]{BlaHan82DimFct}\footnote{One of the longest standing problems in the theory of \ca{s} is determining if every $2$-quasitrace is, in fact, a $[0,\infty]$-valued trace. Haagerup \cite{Haa14Quasitraces} famously showed that this is indeed true whenever the \ca{} is exact.}. We denote by $d_\tau$ its associated \emph{dimension function}, given by $d_\tau (a):=\lim_n \tau (a^{1/n})$. One can show that $d_\tau$ is additive on orthogonal pairs and satisfies $d_\tau (a)\leq d_\tau (b)$ whenever $a\precsim b$. Another important property about dimension functions that we will use repeatedly is \autoref{prp:DFnotStb} below, which we record here for future use. We refer the reader to \cite{BlaHan82DimFct} and \cite[Section~6]{GarPer23arX:ModernCu} for an in-depth discussion on these functions and their relation to functionals on the Cuntz semigroup.

\begin{lma}\label{prp:DFnotStb}
    Let $A$ be a \ca{} and let $d_\tau$ be a dimension function on its stabilization $A\otimes \mathcal{K}$. Assume that $d_\tau (b)\neq 0$ for some $b\in (A\otimes\mathcal{K})_+$. Then, there exists $a\in A_+$ such that $a\precsim b$ and $d_\tau (a)\neq 0$.
\end{lma}

The equivalence between (i) and (ii) in the following proposition appears in \cite[Proposition~13.6]{GarPer23arX:ModernCu}, although one can also prove it directly from the characterizations given in \autoref{thm:MultChar}. We sketch here a slightly different proof of the whole result.

\begin{prp}\label{prp:TracCharNSca}
    Let $A$ be a stable \ca{}. Then, the following are equivalent
    \begin{itemize}
        \item[(i)] $A$ is nowhere scattered;
        \item[(ii)] no dimension function $d_\tau$ satisfies $d_\tau (A_+) = \{0,1,2,\ldots,\infty\}$;
        \item[(iii)] for every $2$-quasitrace $\tau$ on $A$, either $d_\tau (A_+)\subseteq\{0,\infty\}$ or there exists $c\in A_+$ such that $0<d_\tau (c)<1$.
    \end{itemize}
\end{prp}
\begin{proof}[Sketch of the proof]
    Note that (iii) trivially implies (ii). To see that (ii) implies (i), assume for the sake of contradiction that there exist ideals $J\subseteq I\subseteq A$ such that $I/J$ is elementary. Since $A$ is stable, $I/J\cong\K (H)$ with $H$ infinite dimensional. In particular, it follows from the spectral theorem that the composition of this isomorphism with the rank map takes values exactly on $\NN\cup\{\infty\}$. Precomposing with the quotient $I\to I/J$, we get a map $d_I\colon I_+\to \NN\cup\{\infty\}$. Let $d\colon A_+\to \NN\cup\{\infty \}$ be given by $d(a)=d_I (a)$ if $a\in I_+$ and $d(a)=\infty$ otherwise. It follows from Theorem~6.9 and Lemma~13.1 in \cite{GarPer23arX:ModernCu} that $d=d_\tau$ for some $\tau \in {\rm QT}(A)$. This contradicts (ii), since $d_\tau (A_+)=d(A_+)=\mathbb{N}\cup\{\infty \}$.

    Now assume that $A$ is nowhere scattered and let us prove (iii). Thus, let $\tau$ be a $2$-quasitrace such that its associated dimension function $d_\tau$ satisfies $d_\tau (A_+)\not\subseteq \{0,\infty\}$. Set $J=\{a\in A_+\mid d_\tau (a)=0\}$, which is a (proper) ideal by \cite[Theorem~I.1.17]{BlaHan82DimFct}. We note that $A/J$ is nowhere scattered by \autoref{prp:PermProp}, and that the induced dimension function on $A/J$ is faithful and has the same values as $d_\tau$. Thus, upon passing to a quotient, we may assume without loss of generality that $d_\tau$ itself is faithful.

    Let $n\in\NN$ be the least non-negative integer such that $\inf \{d_\tau (A_+\setminus\{0\})\}\in [n,n+1)$, and let $b\in A_+$ be such that $d_\tau (b)\in [n,n+1)$. For any nonzero element $a\in A_+$ such that $a\precsim b$ and for any nonzero $\varepsilon <\Vert a\Vert$, it follows from \cite[Lemma~7.1]{RorWin10ZRevisited}\footnote{This is an early (and weaker) version of what would later be called \axiomO{5}, a property introduced in \cite[Definition~4.1]{AntPerThi18TensorProdCu} that is satisfied in any \ca{}.} that there exist orthogonal positive elements $a_0$ and $c$ such that
    \[    
        b\precsim a_0+c ,\quad 
        a_0\sim a,\andSep 
        d_\tau (b)-d_\tau (a)\leq d_\tau (c)\leq d_\tau (b)-d_\tau ((a-\varepsilon)_+).
    \]

    Note that, since $(a-\varepsilon)_+$ is nonzero and $d_\tau$ is faithful, we have $d_\tau ((a-\varepsilon)_+)\geq n$. Thus, we get $d_\tau (c)<(n+1)-n=1$. Further, if we had $d_\tau (a)<d_\tau (b)$, we would also get $d_\tau (c)>0$. In other words, if there exists some nonzero $a\in A_+$ such that $a\precsim b$ and $d_\tau (a)<d_\tau (b)$, we can produce an element $c$ with the desired properties.

    We claim that such an element $a$ always exists. Indeed, if no such $a$ exists, then $d_\tau (b)=d_\tau (a)$ for all $a\precsim b$. By \cite[Lemma~7.1]{RorWin10ZRevisited}, we can obtain elements $a_0,c$ as above for any $\varepsilon\in (0,\Vert a\Vert)$. Since $(a-\varepsilon)_+\precsim b$ and $(a-\varepsilon)_+$ is not zero, then $d_\tau ((a-\varepsilon)_+)=d_\tau (b)$. Thus, $d_\tau (c)=0$ and, since $d_\tau$ is faithful, we must have $c=0$ and hence $a\precsim b\precsim a_0+c=a_0\sim a$. In other words, $b\sim a$. This shows that $b$ is a $\precsim$-minimal element.

    One can now show that the (simple) ideal generated by $b$ is elementary (cf. \cite[Lemma~8.2]{ThiVil24NowhereScattered}), which contradicts our assumption of nowhere scatteredness.   
\end{proof}

In short, property (iii) above says that, in order to prove nowhere scatteredness, one need not look at all the values $d_\tau (a)$ for each $\tau$ and $a$. Instead, it is enough to produce an element $c$ for every quasitrace $\tau$ such that $d_\tau (c) \in (0,1)$ whenever $d_\tau$ has nonzero finite values. This flexibility is crucial when dealing with algebras $A$ where we may not know how to deal with general positive elements but nevertheless have a good understanding of a specific subset of $A_+$.

As a first example, one can study nowhere scatteredness of multiplier algebras by solely studying the subset of positive elements given by strictly convergent sums of pairwise orthogonal operators; see \autoref{thm:MA} below. In fact, we have already encountered another instance in these notes where (iii) eases the proof of nowhere scatteredness. Indeed, the easiest way to prove that minimal tensor products work well with nowhere scatteredness (\autoref{prp:PermProp}) is to use Kirchberg's slice lemma and show that for any quasitrace $\tau$ there is an elementary tensor $a\otimes b$ with $a\in A_+$ and $b\in B_+$ such that $0<d_\tau (a\otimes b)<1$.

Recall that a \ca{} is said to have \emph{real rank zero} (resp. \emph{stable rank one}) if self-adjoint invertible elements (resp. invertible elements) in the minimal unitization $\tilde{A}$ of $A$ are dense in $\tilde{A}_{sa}$ (resp. $\tilde{A}$) \cite{BroPed91CAlgRR0,Rie83DimSRKThy}. The \emph{nuclear dimension} of a \ca{} is a \ca{ic} analogue of Lebegue's topological covering dimension \cite{WinZac10NuclDim}.

\begin{thm}[cf. {\cite[Theorem~1.4]{Vil23NScaMult}}]\label{thm:MA}
    Let $A$ be a $\sigma$-unital \ca{} of
    \begin{enumerate}
        \item finite nuclear dimension; or
        \item real rank zero; or
        \item stable rank one and strict comparison.
    \end{enumerate}

    Then, $A$ is nowhere scattered if and only if its multiplier algebra $\mathcal{M}(A)$ is.
\end{thm} 
\begin{proof}[Idea of the proof]
    Let $A$ be a \ca{} and let $\mathcal{M}(A)$ be its multiplier algebra. It follows from \autoref{prp:PermProp} that $A$ is nowhere scattered whenever $\mathcal{M}(A)$ is. Thus, assume that $A$ is nowhere scattered and let us prove that $\mathcal{M}(A)$ is. Since nowhere scatteredness works well with extensions (\autoref{prp:PermProp}), it suffices to show that the corona algebra $\mathcal{Q}(A)=\mathcal{M}(A)/A$ is nowhere scattered.

    Take a dimension function $d_\tau$ on $\mathcal{Q}(A)\otimes\mathcal{K}$, and assume that $d_\tau$ takes values exactly on $\NN\cup\{\infty\}$. Find $a\in(\mathcal{Q}(A)\otimes\mathcal{K})_+$ such that $d_\tau (a)=1$. \autoref{prp:DFnotStb} lets us choose $a\in \mathcal{Q}(A)_+$. Using that $A$ is $\sigma$-unital, we know that every element in $\mathcal{Q}(A)$ can be written as the sum of two positive (classes of) strictly convergent sums of pairwise orthogonal elements in $A_+$. Applying this to $a$, we deduce that there exists a strictly convergent sum $S$ as before whose value on $d_\tau$ is in $(0,1]$.

    Under any of the stated conditions of the theorem and the assumption of nowhere scatteredness of $A$, one can produce another strictly convergent sum $S_0$ and $m\in\NN$ such that $d_\tau (S_0)\leq \frac{1}{2}d_\tau (S)\leq md_\tau (S_0)$ for some $m\in\NN$; see \cite[Lemma~5.8]{Vil23NScaMult}. In particular, $d_\tau (S_0)\in (0,\frac{1}{2}]$, a contradiction. This shows that $\mathcal{Q}(A)$ is nowhere scattered by \autoref{prp:TracCharNSca}.
\end{proof}

\begin{rmk}
    Similar to the discussion from \autoref{rmk:CX_NSCvsGGP}, it is not known if the analogue of the previous result holds for the Global Glimm Property. The main reason for this is that we currently lack a characterization of the form in \autoref{prp:TracCharNSca}.
\end{rmk}

\subsection{\texorpdfstring{Weakly purely infinite \ca{s}}{Weakly purely infinite C*-algebras}}\label{subsec:WPIAlgs}

The concept of pure infiniteness for simple \ca{s} was first introduced by Cuntz in \cite{Cun81Kth}, offering a $\mathrm{C}^*$-analogue to type III factors. One of Kirchberg’s celebrated `Geneva Theorems' \cite{Kir94ICM} asserts that a unital, simple, separable, nuclear \ca{} is purely infinite precisely when it tensorially absorbs the Cuntz algebra 
\[
    \mathcal{O}_\infty := C^*\left(s_1, s_2, \ldots \mid s_i^*s_i = 1,\, \sum_i s_is_i^* = 1\right).
\]

Since its inception, the notion of simple pure infiniteness ---equivalently $\mathcal{O}_\infty$-absorption in the nuclear case--- has played a crucial role in classification theory \cite{Kir25Book,Phi00ClassNuclPISimple} and has recently gained prominence in the study of $\mathrm{C}^*$-dynamics \cite{GabSza24Dyn}.

Motivated by the significance of pure infiniteness in the simple case, Kirchberg and R\o{}rdam set out to find an appropiate generalization of not necessarily simple pure infinite \ca{s}. This was done in \cite{KirRor00PureInf,KirRor02InfNonSimpleCalgAbsOInfty}, where they introduced three shades of pure infinitiness: strongly purely infinite, purely infinite, and weakly purely infinite \ca{s}. The latter two notions will be the focus of this subsection.

Let $a$ be a positive element in a \ca{} $A$. Following the notation introduced in \autoref{rmk:NtnDiagSum}, we will denote by $a^{\oplus n}$ the diagonal matrix $a\oplus\ldots\oplus a$ in $M_n(A)$.

\begin{dfn}[{\cite[Definition~4.3]{KirRor02InfNonSimpleCalgAbsOInfty}}]\label{dfn:WPI}
    Let $n\in\NN$. A \ca{} $A$ is $n$-\emph{weakly purely infinite} if, for every positive element $a$ in $A$, one has
    \[
        a^{\oplus n}\sim 
        a^{\oplus (n+1)}
    \]
    in $M_{n+1}(A)$.
    
    We say that the algebra $A$ is \emph{weakly purely infinite} if it is $n$-weakly purely infinite for some $n$. If $A$ is $1$-weakly purely infinite, we say that $A$ is \emph{purely infinite}.
\end{dfn}

\begin{prp}
    Let $A$ be a weakly purely infinite \ca{}. Then, $A$ is nowhere scattered.
\end{prp}
\begin{proof}
    Let $n\in\NN$ be such that $A$ is $n$-weakly purely infinite, and let $d_\tau$ be a dimension function on its stabilization $A\otimes\mathcal{K}$. 
    
    Assume that $d_\tau$ takes some nonzero finite value. In particular, $d_\tau$ must also take some nonzero finite value when restricted to $A_+$ (\autoref{prp:DFnotStb}). However, one has 
    \[
        nd_\tau(a)=(n+1)d_\tau (a)
    \]
    for any $a\in A_+$, which shows that $d_\tau (A_+)\subseteq \{ 0,\infty\}$. This proves that no dimension function on $A\otimes\mathcal{K}$ has a nonzero finite value and, consequently, that $A$ is nowhere scattered as a combination of Propositions \ref{prp:PermProp} and \ref{prp:TracCharNSca}.

    Alternatively, one could have used the fact that any $n$-weakly purely infinite \ca{} has a $n^2$-weakly purely infinite stabilization \cite[Proposition~4.5]{KirRor02InfNonSimpleCalgAbsOInfty}.
\end{proof}

One of the most important problems in the theory of purely infinite \ca{s} is to determine if all three notions introduced by Kirchberg and R\o{}rdam agree. Specialized to our case, the problem reads as follows:

\begin{qst}[{\cite[Question~9.5]{KirRor02InfNonSimpleCalgAbsOInfty}}]\label{qst:WPIProb}
    Is every weakly purely infinite \ca{} purely infinite?
\end{qst}

\cite[Proposition~4.11]{KirRor02InfNonSimpleCalgAbsOInfty} shows that $\sigma$-unital, weakly purely infinite \ca{s} have a weakly purely infinite multiplier algebra. Another question of interest ---which is likely more approachable--- specializes the previous question to this case:

\begin{qst}
    Let $A$ be a $\sigma$-unital purely infinite \ca{}. Is $\mathcal{M}(A)$ purely infinite?
\end{qst}

We will continute the discussion on the weakly purely infinite problem (\autoref{qst:WPIProb}) in \autoref{subsec:PIAlgs}.

\subsection{Groups, graphs and cocycles}
We end this section by making some remarks on nowhere scatteredness for \ca{s} built out of combinatorial data, such as (reduced) group \ca{s}, twisted group \ca{s} or graph algebras. In general, an interesting question is to determine conditions on the underlying objects that characterize nowhere scatteredness of the algebras. The relevance of this study lies, once again, on dimension-reduction phenomena (\autoref{subsec:DimRed}); see also \autoref{pgr:GaborFrames}.

\begin{qst}
    For \ca{s} built from `combinatorial data' (e.g. groups, graphs, groupoids), can one characterize nowhere scatteredness in terms of a property of the underlying object?
\end{qst}

To set the stage for \autoref{subsec:DimRed}, we point out the following fact: In \autoref{prp:NWScaGraph} (resp. \autoref{thm:NwSCTwist}) we state that nowhere scatteredness and $\mathcal{Z}$-stability coincide for a certain class of graph \ca{s} (resp. twisted group \ca{s}). In both of these cases, it is known that the algebras have finite nuclear dimension.

The following is Theorem~B of the recent preprint \cite{Fau25ZstabGraAlg}. Recall that a graph  has \emph{condition (K)} if every vertex on a cycle has two distinct return paths \cite[Section~6]{KPRR97}.

\begin{thm}\label{prp:NWScaGraph}
    Let $E$ be a finite graph. Then, the following are equivalent:
    \begin{itemize}
        \item[(i)] $C^*(E)$ is nowhere scattered;
        \item[(ii)] $C^*(E)$ is $\mathcal{Z}$-stable;
        \item[(iii)] $E$ has no sources and satisfies condition (K).
    \end{itemize}
\end{thm}

\begin{conj}[{\cite[Conjecture~4.1]{Fau25ZstabGraAlg}}]
    For a row-finite graph $E$, $C^* (E)$ is nowhere scattered if and only if $C^*(E)$ is $\mathcal{Z}$-stable, if and only if $E$ has Condition (K) and distinct detours.
\end{conj}

The notion of nowhere scatteredness in the context of reduced group \ca{s} remains relatively unexplored. It is well known that the underlying group must be non-amenable ---otherwise, the left-regular representation gives rise to an elementary quotient--- but this appears to be the only known necessary condition. Nevertheless, examples of nowhere scattered group \ca{s} can be readily constructed using the permanence properties from \autoref{prp:PermProp}. For instance, the reduced group \ca{} of $G\times H$, where $G$ is any exact group and $H$ is any nontrivial $\mathrm{C}^*$-simple group, is nowhere scattered since $C^*_r(G\times H)\cong C^*_r(G)\otimes C_r^*(H)$.

In the context of (reduced) twisted group \ca{s}, one can say much more. This is partly because, unlike the untwisted case, many twisted group \ca{s} of amenable groups can be $\mathcal{Z}$-stable. Indeed, the noncommutative tori discussed in \autoref{exas:NowScat} arise as twisted group \ca{s} of the form $C_r^* (\mathbb{Z}^{2d},\sigma)$. Building upon this example and the ideas from \cite{EckMcK18DimNucNilpotGps}, one can characterize which $2$-cocycles yield nowhere scattered \ca{s} of finitely generated nilpotent groups.

\begin{thm}[{\cite[Theorem~4.6]{EnsVil25arX:Twisted}}]\label{thm:NwSCTwist}
    Let $G$ be a discrete finitely generated nilpotent group, and let $\sigma$ be a $2$-cocycle on $G$. Then, the following are equivalent:
    \begin{itemize}
        \item[(i)] $C^*_r(G,\sigma)$ is nowhere scattered;
        \item[(ii)] $C^*_r(G,\sigma)$ is $\mathcal{Z}$-stable;
        \item[(iii)] $\sigma$ is nonrational.
    \end{itemize}
\end{thm}

We will not discuss here the general definition of non-rationality of $\sigma$. Instead, let us give its definition when $G$ is abelian: In this setting, $\sigma$ is nonrational if the index $[G:Z(G,\sigma )]$ is infinite, where 
\[
    Z(G,\sigma)=
    \{
    g\in G\mid \sigma (g,h)=\sigma(h,g)\,\forall h\in G
    \}.
\]
When $G$ is $\mathbb{Z}^{2d}$ for some $d\geq 1$ ---in which case $C^*(\mathbb{Z}^{2d},\sigma )\cong A_\Theta$ for some $\Theta$--- the aforementioned condition is precisely the same as $\Theta$ being nonrational; see e.g. \cite[Proposition~2.1]{EnsThiVil25Criteria}. We refer the reader to \cite{EnsVil25arX:Twisted} for the general definition of non-rationality and the proof of \autoref{thm:NwSCTwist}.

\begin{pgr}[Gabor frames and regularity properties]\label{pgr:GaborFrames}
Beyond the intrinsic interest of characterizing nowhere scatteredness or $\mathcal{Z}$-stability for twisted group \ca{s}, results such as \autoref{thm:NwSCTwist} are surprisingly related to time-frequency analysis and other types of sampling problems.

Namely, a number of these problems can be translated into asking if a pair of distinguished projections $p,q$ in $C_r^*(G,\sigma)\otimes\mathcal{K}$ satisfy $p\precsim q$.\footnote{The Cuntz subequivalence applied to projections reduces to the well-known \emph{Murray-von Neumann subequivalence}, defined as $p\precsim q$ whenever $p=vv^*$ and $v^*v\leq q$ for some $v$.} In the specific case of \cite{EnsVil25arX:Twisted}, the goal was to obtain a partial converse to the Balian-Low theorem for lattices $G$ in  connected, simply connected, nilpotent Lie groups. By the results in \cite[Theorem~6.6]{BedEnsvVelt22} (and some crucial observations from \cite[Theorem~5.1]{EnsVil25arX:Twisted}), there exists a distinguished projection $p\in C_r^*(G,\sigma)\otimes\mathcal{K}$ such that proving the desired converse is equivalent to answering the question: Is it true that $\tau(p)<\tau(1_{C_r^*(G,\sigma)})$ for every trace $\tau$ implies $p\precsim 1_{C_r^*(G,\sigma)}$?

The question above has a positive answer whenever $C_r^*(G,\sigma)$ is $\mathcal{Z}$-stable\footnote{These two properties are not equivalent: $C_r^*(\mathbb{Z},1)\cong C(\mathbb{T})$ is not $\mathcal{Z}$-stable but the question has a positive answer.}; see \autoref{subsec:DimRed}. Thus, combining \autoref{thm:NwSCTwist} with this operator algebraic translation, one obtains the following result in generalized time-frequency analysis (\cite[Theorem~5.1]{EnsVil25arX:Twisted}):

\emph{Let $G$ be a lattice of a connected, simply connected, nilpotent Lie group $H$, and let $\pi$ be a $\sigma$-projective, square-integrable, irreducible, unitary representation of $H$. Assume that $\sigma$ is non-rational on $G$. Then, there exists a frame $\pi (G)\gamma$ with $\gamma$ smooth whenever $d_\pi {\rm covol}(G)<1$.}
\end{pgr}

\section{The Global Glimm Property}

We now move to the second notion in the Global Glimm Problem. After stating its definition and main permanence properties, we give a brief proof that the Global Glimm Property implies nowhere scatteredness.

Given elements $s,r\in A$, we will write $s \lhd r$ if $s$ belongs to the closed ideal generated by $r$.

\begin{dfn}\label{dfn:GGP}
    A \ca{} $A$ has the \emph{Global Glimm Property} if for every $a\in A_+$ and every $\varepsilon>0$ there exists $r\in\overline{aAa}$ such that $r^2=0$ and $(a-\varepsilon)_+\lhd r$.
\end{dfn}

\begin{rmk}
    \autoref{dfn:GGP} appeared in \cite[Definition~1.2]{BlaKir04GlimmHalving} under the name \emph{Global Glimm Halving Property}.

    The original definition of the Global Glimm Property as given in \cite[Definition~4.12]{KirRor02InfNonSimpleCalgAbsOInfty} is formally different: In their paper, Kirchberg and R\o{}rdam say that a \ca{} satisfies the property if for every $a\in A_+$, every $\varepsilon>0$, and every integer $k\geq 2$, there exists a $^*$-homomorphism $M_k(C_0((0,1]))\to \overline{aAa}$ such that $(a-\varepsilon)_+$ is contained in the ideal generated by its image.

    One can show that these two formulations are indeed equivalent; see \cite[Theorem~3.6]{ThiVil23Glimm}.
\end{rmk}

\begin{exas}\label{exas:GGP} $ $
\begin{enumerate}
    \item Simple non-elementary \ca{s} have the Global Glimm Property. This is no longer a trivial consequence of the definition (as in \autoref{exas:NowScat}). Instead, it requires a proof, which can be extracted from Glimm's \cite[Lemma~4]{Gli61Type1}. It is precisely this result that gives the property its name.
    \item $\mathcal{Z}$-stable algebras have the Global Glimm Property. To see this, recall that every corner $\overline{aAa}$ of a $\mathcal{Z}$-stable algebra $A$ is again $\mathcal{Z}$-stable \cite[Corollary~3.1]{TomWin07ssa}. Thus, one can obtain a full square-zero element $r$ from $\mathcal{Z}$ (because it is simple and non-elementary) and view it in $\overline{aAa}\cong\overline{aAa}\otimes\mathcal{Z}$ as $a\otimes r$. Recall that a simple tensor is square-zero whenever one of its components is, and so $a\otimes r$ is square-zero.
\end{enumerate}
\end{exas}

\begin{prp}[{\cite[3.10-3.12]{ThiVil23Glimm}}]\label{prp:PermPropGGP} The following permanence properties hold:
\begin{enumerate}
    \item Morita equivalence: Let $A,B$ be Morita equivalent. Then, $A$ has the Global Glimm Property if and only if $B$ does.
    \item Extensions: The Global Glimm Property passes to hereditary sub-\ca{s} and quotients. Further, if $I$ is an ideal of a \ca{} $A$, then $A$ has the Global Glimm Property if and only if $I$ and $A/I$ do.
    \item Inductive limits: An inductive limit of \ca{s} with the Global Glimm Property has the Global Glimm Property.
\end{enumerate}
\end{prp}

As stated in the introduction, the Global Glimm Property always implies nowhere scatteredness. Since this is a crucial point of this survey, we provide a short proof here for the convenience of the reader.

\begin{lma}\label{prp:GGPimpNSCa}
    Let $A$ be a \ca{} with the Global Glimm Property. Then, it is nowhere scattered.
\end{lma}
\begin{proof}
    Assume for the sake of contradiction that $A$ is not nowhere scattered, which by \autoref{prp:TracCharNSca} implies that there exists a dimension function $d_\tau$ on $A\otimes\mathcal{K}$ such that $d_\tau(A\otimes\mathcal{K})=\mathbb{N}\cup\{\infty\}$. Now, let $a\in (A\otimes\mathcal{K})_+$ be such that $d_\tau(a)=1$. By \autoref{prp:DFnotStb}, we may assume that $a\in A_+$.

    Let $\varepsilon>0$ be such that $d_\tau((a-\varepsilon)_+)=1$. Then, since $A$ satisfies the Global Glimm Property, there exists a square-zero element $r$ in $\overline{aAa}$ such that $(a-\varepsilon/2)_+$ is in the ideal generated by $r$. In particular, the positive element $r^*r$ satisfies
    \[
        r^*r\perp rr^*,\andSep 
        rr^*+r^*r\precsim a,
    \]
    because $r\in \overline{aAa}$.

    Thus, we have $2d_\tau(r^*r)=d_\tau(rr^*+r^*r)\leq d_\tau(a)=1$, which implies $d_\tau (r^*r)=0$ by our assumption. In particular, every element in the ideal generated by $r^*r$ lies in the kernel of $d_\tau$. This includes $(a-\varepsilon)_+$, a contradiction.
\end{proof}

\begin{qst}[The Global Glimm Problem]\label{qst:GGP}
    Does every nowhere scattered \ca{} have the Global Glimm Property?
\end{qst}

As noted in \cite[Problem~LXXIII]{SchTikWhi99arX:Prob}, one can rephrase \autoref{qst:GGP} as follows.

\begin{qst}
    Let $A$ be a \ca{} without finite dimensional representations, and let $a\in A_+$ and $\varepsilon>0$. Does there exist a square-zero element $r\in A$ such that $(a-\varepsilon)_+\lhd r$?
\end{qst}

This form of the problem is akin to \cite[Question~2]{EllRor06Perturb}, which asked the question in the unital case.

\subsection{\texorpdfstring{Purely infinite \ca{s}}{Purely infinite C*-algebras}}\label{subsec:PIAlgs}
We continue with the discussion from \autoref{subsec:WPIAlgs}. Namely, we show that any purely infinite \ca{} satisfies the Global Glimm Property, and we relate the Global Glimm Problem to the weakly purely infinite problem. Since we have not proved the equivalence between \autoref{dfn:GGP} and Kirchberg-R\o{}rdam's \cite[Definition~1.2]{BlaKir04GlimmHalving}, we do minimal changes to the proofs of \cite{KirRor02InfNonSimpleCalgAbsOInfty} to adapt them to the definition of the Global Glimm Property provided in these notes.

\begin{lma}[{\cite[Lemma~4.14]{KirRor02InfNonSimpleCalgAbsOInfty}}]\label{lma:PIimpGGP}
    Let $A$ be a purely infinite \ca{}. Then, it has the Global Glimm Property.
\end{lma}
\begin{proof}
    Let $a\in A_+$ and $\varepsilon>0$. By definition, we have $a^{\oplus 2}\precsim a$. In particular, it follows from \autoref{prp:CuSubProp}~(2) that $sas^*=(a-\varepsilon)_+^{\oplus 2}$ for some $s\in M_{2,1} (A)$. In other words, there exist $s_1,s_2$ in $A$ such that
    \[
        \begin{pmatrix}
            s_1 a s_1^* & s_1 as_ 2^* \\
            s_2 a s_1^* & s_2 a s_2^*
        \end{pmatrix} = 
        \begin{pmatrix}
            (a-\varepsilon)_+ & 0\\
            0 & (a-\varepsilon)_+
        \end{pmatrix}.
    \]

    Set $r=a^{1/2}s_1^* s_2 a^{1/2}$. Then, $r^2=0$ and 
    \[
        (a-\varepsilon)_+\sim (a-\varepsilon)_+^3 = 
        (s_1 a s_1^*)(s_2 a s_2^*)(s_1 a s_1^*) = 
        (s_1 a^{1/2})rr^* (s_1 a^{1/2})^*\precsim rr^*.
    \]
    
    This implies that $(a-\varepsilon)_+$ is in the ideal generated by $r$, a square-zero element in $\overline{aAa}$.
\end{proof}

\begin{prp}[{\cite[Proposition~4.15]{KirRor02InfNonSimpleCalgAbsOInfty}}]\label{prp:GGPandPIP}
    A \ca{} is purely infinite if and only if it is weakly purely infinite and satisfies the Global Glimm Property.
\end{prp}
\begin{proof}
    Any purely infinite algebra is weakly purely infinite (by definition) and has the Global Glimm Property (by \autoref{lma:PIimpGGP}).

    Thus, assume that $A$ satisfies the Global Glimm Property and is $n$-weakly purely infinite for some $n> 1$. We will show that $A$ is $(n-1)$-weakly purely infinite, and then an inductive argument shows that $A$ is purely infinite.
    
    Take $a\in A_+$ and let $\varepsilon>0$. Using the Global Glimm Property, find a square-zero element $r\in\overline{aAa}$ such that $(a-\varepsilon)_+$ is in the ideal generated by $r$. In particular, $(a-\varepsilon)_+^{\oplus n}$ is in the ideal generated by $(rr^*)\oplus 0_{n-1}$ in $M_n (A)$. Thus, it follows that there exists some $m\in\NN$ such that $(a-\varepsilon)_+^{\oplus n}\precsim (rr^*)^{\oplus m}$. As $A$ is $n$-weakly purely infinite, we have $(rr^*)^{\oplus m}\sim (rr^*)^{\oplus n}$.

    Now, using that $r$ is square-zero in the second step, we get
    \[
        (a-\varepsilon)_+^{\oplus n}\precsim (rr^*)^{\oplus n}\sim (rr^*)^{\oplus (n-1)}+(r^*r)\oplus 0_{n-1}.
    \]

    By construction, each of the two summands is contained in the hereditary sub-\ca{} of $a^{\oplus (n-1)}$, which implies that
    \[
         (a-\varepsilon)_+^{\oplus n}\precsim (rr^*)^{\oplus (n-1)}+(r^*r)\oplus 0_{n-1}\precsim a^{\oplus (n-1)}.
    \]

    Since this holds for any given $\varepsilon$, it follows that $a^{\oplus n}\precsim a^{\oplus (n-1)}$, as desired.
\end{proof}

\begin{cor}
    The following statements are equivalent:
    \begin{itemize}
        \item[(i)] The weakly purely infinite problem has a positive answer.
        \item[(ii)] The Global Glimm Problem has a positive answer for weakly purely infinite \ca{s}.
    \end{itemize}
\end{cor}

\section{Applications and known results}\label{sec:ResAndAppl}

This section is divided in two disjoint parts. First, we discuss some dimension reduction phenomena and focus in particular on the so-called \emph{non-simple Toms-Winter conjecture} and the \emph{ranks problem}. We then list known results of the Global Glimm Problem.

\subsection{Dimension reduction phenomena and soft operators}\label{subsec:DimRed}

As mentioned in the introduction, dimension reduction phenomena are what motivated the definitions of both nowhere scatteredness and the Global Glimm Property. We have already encountered a dimension reduction question in these notes: the weakly purely infinite problem, which asks whether $n$-weak pure infiniteness implies $1$-weak pure infiniteness. In the setting of simple \ca{s}, one of the most famous problems related to such phenomena is the \emph{Toms-Winter conjecture}. To introduce it, we first have to recall the notion of \emph{pureness} and \emph{strict comparison}.

\begin{pgr}[Pure \ca{s}]
Following \cite{CowEllIva08CuInv}, recall that the \emph{Cuntz semigroup} of $A$ is given by the quotient $\Cu (A):=(A\otimes\mathcal{K})_+/\sim$. In this set, we define a \emph{partial order} by writing $[a]\leq [b]$ whenever $a\precsim b$. Using that there is a unique isomorphism $\phi\colon \K\otimes M_2\to\K$ up to unitary equivalence, we also equip $\Cu (A)$ with the (well-defined) \emph{addition} $[a]+[b]:=[{\rm id}_A\otimes\phi(a\oplus b)]$.

Given elements $[a],[b]\in\Cu (A)$, we write $[a]\ll [b]$ if there exists $\varepsilon>0$ such that $a\precsim (b-\varepsilon)_+$.\footnote{This relation is known as \emph{compact containment}, and can be characterized intrinsically in terms of the order and suprema of increasing sequences in the Cuntz semigroup. Such characterization plays a crucial role when developing the theory of abstract Cuntz semigroups, also known as \emph{$\Cu$-semigroups}.} The following definitions are from \cite{Win12NuclDimZstable}, with slight modifications from \cite{RobTik17NucDimNonSimple}:
\begin{itemize}
    \item We say that $\Cu (A)$ is \emph{almost divisible} if, for every $n\in\NN$ and for every pair $x',x\in\Cu (A)$ such that $x'\ll x$, there exists $d\in\Cu (A)$ such that
    \[
        nd\leq x,\andSep x'\leq (n+1)d,
    \]
    that is, $d$ is `almost' $\frac{1}{n}x$.
    \item $\Cu (A)$ is \emph{almost unperforated} if, whenever $x,y\in\Cu (A)$ satisfy $(n+1)x\leq ny$ for some $n\in\NN$, one has $x\leq y$.
    \item A \ca{} $A$ (and its Cuntz semigroup $\Cu (A)$) is said to be \emph{pure} if $\Cu (A)$ is both almost unperforated and almost divisible.
\end{itemize}

As shown in \cite[Proposition~6.2]{EllRobSan11Cone}, $\Cu (A)$ is almost unperforated if, and only if, for every pair of elements $a,b\in (A\otimes\mathcal{K})_+$, one has
\[
\left.
\begin{array}{l}
    a\text{ is in the ideal generated by }b\text{; and} \\
    d_\tau (a)<d_\tau(b) \text{ for all }\tau\in{\rm QT}(A)\text{ such that }d_\tau(b)=1
\end{array}
\right\}
\implies 
a\precsim b,
\]
where recall that ${\rm QT}(A)$ denotes the set of $[0,\infty]$-valued lower semicontinuous $2$-quasitraces.

This condition is referred to as \emph{strict comparison} (of positive elements by quasitraces) in the non-simple case. In the case of simple \ca{s}, the first condition can be dispensed with.
\end{pgr}

\begin{pgr}[The conjecture]
The Toms–Winter conjecture emerged within the context of the \emph{Elliott classification program}. While we have deliberately refrained from discussing this foundational program which has shaped the theory of \ca{s} over the past several decades, we strongly encourage interested readers to consult the excellent ICM proceedings on the subject \cite{Ror07ICM, Whi23ICM, Win18ICM}. The conjecture itself originates from a paper by Toms and Winter \cite{TomWin09Villadsen}, and poses the question of whether three key regularity properties coincide.

\begin{conj}[The Toms-Winter conjecture]
    For a unital, separable, simple, nuclear, non-elementary \ca{} $A$, the following conditions are equivalent:
    \begin{itemize}
        \item[(i)] $A$ has finite nuclear dimension;
        \item[(ii)] $A$ is $\mathcal{Z}$-stable;
        \item[(iii)] $A$ has strict comparison.
    \end{itemize}
\end{conj}

By now, this conjecture is almost a theorem: R\o{}rdam proved that any $\mathcal{Z}$-stable algebra (simple or not) has strict comparison in \cite[Theorem~4.5]{Ror04StableRealRankZ}; Winter showed in \cite{Win12NuclDimZstable} that (i) implies (ii); and Castillejos, Evington, Tikuisis, White and Winter proved the converse in \cite{CasEviTikWhiWin21NucDimSimple}. In other words, all that remains is (iii) implies (ii), although there are many partial results on the subject \cite{KirRor14CentralSeq,MatSat12StrComparison,Sat12arx:TraceSpace,Thi20RksOps,TomWhiWin15ZStableFdBauer}. A surprising and extremely deep fact that follows by combining \cite{CasEviTikWhiWin21NucDimSimple,Win12NuclDimZstable} is that a unital separable simple (non-elementary) \ca{} has finite nuclear dimension if and only if it has nuclear dimension at most $1$. This, once again, is a dimension reduction phenomenon.

Motivated in part by the considerable attention the Toms–Winter conjecture has garnered over the past decade, it is natural to ask whether some of these notions also coincide in the setting of nowhere scattered \ca{s}. This is a necessary condition to assume, since $\mathcal{Z}$-stability implies nowhere scatteredness (\autoref{exas:NowScat}).

\begin{qst}\label{qst:NonSimpTW}
    Let $A$ be a unital, separable, nuclear, nowhere scattered \ca{}. Are the following equivalent?
    \begin{itemize}
        \item[(i)] $A$ has finite nuclear dimension.
        \item[(ii)] $A$ is $\mathcal{Z}$-stable.
        \item[(iii.1)] $A$ is \emph{Cuntz semigroup regular} \cite[Problem~LXXV]{SchTikWhi99arX:Prob}, that is, the first factor map $A\to A\otimes\mathcal{Z}$ given by $x\mapsto x\otimes 1$ induces an isomorphism $\Cu (A)\cong\Cu (A\otimes\mathcal{Z})$.
        \item[(iii.2)] $A$ is pure.
        \item[(iii.3)] $A$ has strict comparison.
    \end{itemize}
\end{qst}

Conditions (iii.1)–(iii.3) each capture some aspect of regularity within the Cuntz semigroup, and determining whether they are equivalent is, in itself, a question of independent interest (also in the simple case). Beyond the trivial downward implications, it is known that (iii.2) and (iii.3) coincide under the assumption of stable rank one \cite{AntPerRobThi22CuntzSR1}. Further, (iii.1) and (iii.2) are equivalent for unital stably finite simple \ca{s}. The only proof of this fact so far (\cite[Proposition~22]{SchTikWhi99arX:Prob}) uses that pureness implies stable rank one for such \ca{s} \cite{Lin25Sr1}.
\end{pgr}

Let $\omega$ be a free ultrafilter on $A$, and let $A_\omega$ denote its ultraproduct \cite[II.8.1.7]{Bla06OpAlgs}. Recall that the \emph{central sequence algebra of $A$} is defined as $F(A):=(A'\cap A_\omega)/(A^\perp\cap A_\omega)$. This construction, introduced by Kirchberg in \cite[Definitions~1.1]{Kir06CentralSeqPI}, has become a powerful tool for analyzing regularity properties of \ca{s}. Of particular interest to us is the connection of $F(A)$ and the implication `(i)$\implies$(ii)' in \autoref{qst:NonSimpTW}. More concretely, one can use results from \cite{RobTik17NucDimNonSimple} to reduce this implication to an instance of the Global Glimm Problem:

\begin{thm}\label{prp:RobTikGGP}
    Let $A$ be a separable \ca{} of finite nuclear dimension. Then,
    \begin{enumerate}
        \item if $A$ is nowhere scattered, so is $F(A)$.
        \item if $F(A)$ has the Global Glimm Property, $A$ is $\mathcal{Z}$-stable.
    \end{enumerate}
\end{thm}
\begin{proof}[Sketch of the proof]
    To prove (1), assume that $A$ is nowhere scattered. Then, it follows as a combination of \cite[Theorem~6.1]{RobTik17NucDimNonSimple} and \cite[Lemma~5.3]{AntPerRobThi24TracesUltra} that, for every $k,N\in\NN$ there exists $c\in F(A)_+$ and $s\in\NN$ such that
    \[
        c^{\oplus k}\precsim 1_{F(A)} , \andSep 
        1_{F(A)}^{\oplus Ns}\precsim (c^{\oplus Nk}\oplus 1_{F(A)})^{\oplus s}.
    \]

    Let $C$ be a separable sub-\ca{} of $F(A)$ containing $1_{F(A)}$, $c$, and all the elements required to witness the previous Cuntz subequivalences.\footnote{For example, for every $i\in\NN_{>0}$ the subalgebra $C$ has to contain all the entries of some  matrices $M_i$ such that $\Vert c^{\oplus k}- M_i(1_{F(A)}\oplus 0_{k-1})M_i^*\Vert\leq \frac{1}{i}$.} For any $a\in F(A)_+$, we know from \cite[Proposition~1.12]{Kir06CentralSeqPI} that there exists a unital injective $^*$-homomorphism $\varphi\colon C\to F(A)\cap C^*(a)'$.

    Since $a$ commutes with $\varphi (C)$, we can `multiply' the previous Cuntz subequivalences by a diagonal of $a$'s to obtain
    \[
        (a\varphi(c))^{\oplus k}\precsim a , \andSep 
        a^{\oplus Ns}\precsim ((a\varphi(c))^{\oplus Nk}\oplus a)^{\oplus s}.
    \]

    This shows that, for every $k,N\in\NN$ and every $a\in F(A)_+$, there exists $e\in F(A)_+$ such that
    \[
        kd_\tau (e)\leq d_\tau (a),\andSep
        d_\tau (a)\leq k (d_\tau(e))+\frac{1}{N}d_\tau (a)
    \]
    for every $\tau \in {\rm QT}(F(A))$.

    Now fix $\tau\in {\rm QT}(F(A))$. By \autoref{prp:DFnotStb}, either $d_\tau (F(A)_+)\subseteq \{ 0,\infty\}$ or there exists $a\in F(A)_+$ such that $d_\tau (a)$ is nonzero and finite. In the second case, find $l\in\NN$ such that $d_\tau (a)<l$ and let $e$ be as above (for $k=l$ and any choice of $N>1$). Using that $d_\tau (a)$ is finite and nonzero, we get $0<d_\tau (e)<1$, as desired.

    For (2), apply the Global Glimm Property to the unit in $F(A)$ and $\varepsilon<1$. This gives a full square-zero element $r$. In particular $rr^*$ and $r^*r$ are orthogonal full positive elements in $F(A)$. The result now follows from \cite[Theorem~1.2]{RobTik11HilbCommutative}.\footnote{Kirchberg and R\o{}rdam showed in \cite[Proposition~6.1]{KirRor15CentralSeqCharacters} that the existence of two orthogonal full positive element in $F(A)$ is equivalent to the \emph{Glimm Halving Property}, a condition that is often much weaker than the Global Glimm Property. However, it will follow from \autoref{prp:ConRobTik} that these conditions do agree on central sequence algebras.}
\end{proof}

\begin{cor}\label{cor:EqTWGGP}
    The following statements are equivalent:
    \begin{itemize}
        \item[(i)] Every separable, nowhere scattered \ca{} of finite nuclear dimension is $\mathcal{Z}$-stable.
        \item[(ii)] The Global Glimm Problem has a positive answer for central sequence algebras of separable, nowhere scattered \ca{s} of finite nuclear dimension. 
    \end{itemize}
\end{cor}

\begin{rmk}
    The implications (1) and (2) in \autoref{prp:RobTikGGP} are in fact  equivalences. We postpone this proof until \autoref{prp:ConRobTik}, where this will follow easily from Cuntz semigroup methods.
\end{rmk}

As discussed in \autoref{pgr:GaborFrames}, there are numerous cases where one only needs to understand if a \ca{} has strict comparison or not. In the context of the non-simple Toms-Winter conjecture, the most general result thus far is \autoref{thm:PureCalgs} below.

\begin{thm}[{\cite[Theorem~B]{AntPerThiVil24arX:PureCAlgs}}]\label{thm:PureCalgs}
    Let $A$ be a \ca{} of finite nuclear dimension and with the Global Glimm Property. Then, $A$ is pure. In particular, $A$ has strict comparison.
\end{thm}

Let us now briefly discuss the implication from (iii.3) to (iii.2), which asks if strict comparison implies almost divisibility on the Cuntz semigroup whenever the underlying \ca{} is nowhere scattered. One of the modern approaches to this question goes through the \emph{rank problem} (\autoref{qst:RankProb}). To the best of our knowledge, \autoref{qst:RankProb} was first asked by Nate Brown and later studied in several papers; see, for example, \cite{DadTom10Ranks,Thi20RksOps,AntPerRobThi22CuntzSR1}. For a positive element $a$, we define its \emph{rank} as the function ${\rm Rank}(a)\colon \tau\mapsto d_\tau (a)$.\footnote{The rank of a positive element $a\in A_+$ is invariant under Cuntz equivalence and induces the \emph{rank} of $[a]$, which is often denoted by $\widehat{[a]}$.}

\begin{qst}[The rank problem]\label{qst:RankProb}
    Let $A$ be a stable \ca{}. For which lower-semicontinuous affine maps $f\colon {\rm QT}(A)\to [0,\infty ]$ does there exist $a\in A_+$ such that $f(\tau)=d_\tau(a)$ for each $\tau\in {\rm QT}(A)$?
\end{qst}

What will be most relevant to us is the following specialization of the rank problem:

\begin{qst}[Divisibility of ranks]\label{qst:DivRank}
    Let $A$ be a stable \ca{}. Given $\alpha\in (0,1)$ and $a\in A_+$, does there exist $b\in A_+$ such that $d_\tau (b)=\alpha d_\tau (a)$ for every $\tau\in {\rm QT}(A)$?
\end{qst}

We will say that \emph{all ranks are divisible on $A$} if \autoref{qst:DivRank} has a positive answer. Note that nowhere scatteredness is necessary if we want to divide all ranks, since the absence of nowhere scatteredness implies the existence of a dimension function $d_\tau$ that only takes values in $\NN\cup\{\infty\}$ (\autoref{prp:TracCharNSca}).

\autoref{prp:RankProb} below shows the exact relation between \autoref{qst:DivRank} and the implication from (iii.3) to (iii.2). Locating the first reference for this is a bit tricky, but a proof of the result can essentially be found in \cite[Proposition~2.11]{Thi20RksOps} or \cite[Proposition~5.2]{CasEviTikWhi22UniPropGamma}.

\begin{thm}\label{prp:RankProb}
    Let $A$ be a stable \ca{}. Then, the following are equivalent:
    \begin{itemize}
        \item[(i)] $A$ is pure;
        \item[(ii)] $A$ has strict comparison and all ranks are divisible.
    \end{itemize}
\end{thm}
\begin{proof}[Idea of the proof]
    Assume (ii) and let $a\in A_+$ and $n\in\NN$. If all ranks are divisible, one can construct a `functional divisor' of $a$ by considering the element $b\in A_+$ such that ${\rm Rank}(b)=\frac{2}{2n+1}{\rm Rank}(a)$. Using strict comparison, one sees that $b$ `almost divides' $a$ by $n$, as desired.

    That (i) implies (ii) follows from the fact that the Cuntz semigroup of a pure \ca{} has $\Cu (\mathcal{Z})$-multiplication \cite[Theorem~7.3.11]{AntPerThi18TensorProdCu}.
\end{proof}

An important class of elements that often appears whenever one works on the rank problem is that of \emph{soft elements}. As we state formally below, these elements are \emph{abundant} in any \ca{} with the Global Glimm Property. In contrast, we do not know if the same is true for nowhere scattered \ca{s} and, in fact, this condition sits in the middle of these two properties; see \autoref{thm:AbSoft}.

Soft elements were introduced in \cite[Definition~3.1]{ThiVil23arX:Soft}. The following definition is equivalent by \cite[Proposition~3.6]{ThiVil23arX:Soft}.

\begin{dfn}
    Let $A$ be a \ca{} and let $a\in A_+$. We say that $a$ is \emph{soft} if, for every closed ideal $I\subseteq A$, either $0$ is a limit point of the spectrum of $a+I\in A/I$ or $a\in I$.

    A general element $s\in A$ is \emph{soft} if $ss^*$ (equivalently $s^*s$) is soft.
\end{dfn}

\begin{rmk}
A `quotient-free' definition of softness for positive elements is as follows: An element $a\in A_+$ is soft if and only if, for every $\varepsilon>0$, there exists $b\in\overline{aAa}_+$ such that $a\lhd b$ and $(a-\varepsilon)_+\perp b$; see \cite[Proposition~3.6]{ThiVil23arX:Soft}.

Using this characterization, one can produce a soft element from any sequence of pairwise orthogonal positive elements $(a_n)_n$ by considering $\sum_n \frac{1}{2^n\Vert a_n\Vert}a_n$.
\end{rmk}

\begin{dfn}[{\cite[Definition~5.2]{ThiVil23arX:Soft}}]
    A \ca{} has an \emph{abundance of soft elements} if, for every $a\in A_+$ and $\varepsilon>0$, there exists a soft element $b\in\overline{aAa}_+$ such that $(a-\varepsilon)_+\lhd b$.
\end{dfn}

\begin{prp}[{\cite[Propositions~7.4,~7.7]{ThiVil23arX:Soft}}]\label{thm:AbSoft}
    Let $A$ be a \ca{}. Then,
    \begin{enumerate}
        \item if $A$ has an abundance of soft elements, $A$ is nowhere scattered;
        \item if $A$ has the Global Glimm Property, $A$ has an abundance of soft elements.
    \end{enumerate}
\end{prp}

\begin{thm}[{\cite[Theorems~A~and~C]{AsaThiVil23arX:RksSoftOps}}]
    Let $A$ be a stable \ca{} with the Global Glimm Property, and let $a\in A_+$. Then, there exists a soft element $b\in A_+$ such that $b\precsim a$ and $\widehat{b}=\widehat{a}$.

    If, additionally, $A$ has strict comparison, there is a unique such soft element up to Cuntz equivalence.
\end{thm}

In particular, the first part of this theorem reduces the study of certain invariants (such as the radius of comparison or the Cuntz covering dimension) to the soft part of $A$.

\subsection{Partial solutions}

Below is a list of partial solutions for the Global Glimm Problem. All known results follow a similar pattern, obtaining a positive answer by imposing some form of low or finite dimensionality.

\begin{pgr}[Conditions on the primitive ideal space]
Elliott and R\o{}rdam showed in \cite[Corollary~7]{EllRor06Perturb} that the Global Glimm Problem has a positive answer for unital real rank zero \ca{s}. Using the results from \autoref{sec:CuntzSgp} below, one can push this to \ca{s} of \emph{topological dimension zero} (a condition that is satisfied by any real rank zero \ca{}).

\begin{thm}[{\cite[Theorem~2.3]{NgThiVil25arX}}]\label{thm:GGPdim0}
    Let $A$ be a nowhere scattered \ca{} whose primitive ideal space has a basis of compact open subsets. Then, it has the Global Glimm Property.
\end{thm}

If one assumes that the primitive ideal space is Hausdorff, topological dimension zero can be relaxed:

\begin{thm}[{\cite[Theorem~4.3]{BlaKir04GlimmHalving}}]
    Let $A$ be a nowhere scattered \ca{} with finite-dimensional Hausdorff primitive ideal space. Then, it has the Global Glimm Property.
\end{thm}
\end{pgr}

\begin{pgr}[Stable rank one]
The investigation into the Cuntz semigroup of stable rank one \ca{s} \cite{AntPerRobThi22CuntzSR1} provided solutions to several key problems for this class, including the Global Glimm Problem.

\begin{thm}[cf. {\cite[Theorem~9.1]{AntPerRobThi22CuntzSR1}}]\label{thm:GGPsr1}
    Let $A$ be a nowhere scattered \ca{} of stable rank one. Then, it has the Global Glimm Property.
\end{thm}

Both Theorems \ref{thm:GGPdim0} and \ref{thm:GGPsr1} can be understood within the more general approach to the problem explained in \autoref{sec:CuntzSgp}.
\end{pgr}

\begin{pgr}[Nuclear dimension]
Although \autoref{thm:DimNucGGP} below cannot be found in the literature in its stated form, it follows as a direct combination of \cite[Theorem~3.1]{RobTik17NucDimNonSimple}, \cite[Theorem~5.7]{AntPerThiVil24arX:PureCAlgs} and the fact that pure \ca{s} satisfy the Global Glimm Property. Examples covered by the result include nowhere scattered \ca{s} of finite decomposition rank.

\begin{thm}\label{thm:DimNucGGP}
    Let $A$ be a nowhere scattered \ca{} of finite nuclear dimension. Assume that $A$ has no nonzero simple purely infinite quotients. Then, $A$ is pure. In particular, it has the Global Glimm Property.
\end{thm}

The non-simple Toms-Winter conjecture asks, in particular, if nowhere scattered \ca{s} of finite nuclear dimension are pure. This is equivalent (by \autoref{thm:PureCalgs}) to the following question:

\begin{qst}\label{qst:NowScatFinDimGGP}
    Let $A$ be a nowhere scattered \ca{} of finite nuclear dimension. Does $A$ have the Global Glimm Property?
\end{qst}
\end{pgr}

\section{A Cuntz semigroup translation}\label{sec:CuntzSgp}

We have postponed until this last section the discussion on how modern Cuntz semigroup theory enters the picture of the Global Glimm Problem, albeit those familiar with the field will have undoubtedly recognized its presence underlying many of the results presented thus far. Here, we offer a translation of the Global Glimm Problem through the lens of the Cuntz semigroup ---a perspective that has proven instrumental in several recent partial resolutions. Notably, Theorems \ref{thm:GGPdim0} and \ref{thm:GGPsr1} will emerge naturally from the framework developed in this section; see \autoref{cor:GGPsr1dim0}.

\begin{dfn}[{\cite[Definition~5.1]{RobRor13Divisibility}}]\label{dfn:Div}
    Let $k\in\NN$ and let $A$ be a \ca{}. We say that the Cuntz semigroup $\Cu (A)$ of $A$ is
    \begin{itemize}
        \item \emph{weakly $(k,\omega)$-divisible} if, for every $x',x\in\Cu (A)$ such that $x'\ll x$, there exist $n\in\NN$ and $d_1,\ldots ,d_n\in\Cu (A)$ such that 
        \[
            x'\leq d_1+\ldots +d_n
            \andSep 
            kd_j\leq x\text{ for each }j.
        \]
        \item \emph{$(k,\omega)$-divisible} if, for every $x',x\in\Cu (A)$ such that $x'\ll x$, there exist $n\in\NN$ and $d\in\Cu (A)$ such that 
        \[
            d'\leq ny
            \andSep 
            ky\leq d.
        \]
    \end{itemize}
\end{dfn}

Combining the results \cite[Theorem~8.9]{ThiVil24NowhereScattered} and \cite[Theorem~3.6]{ThiVil23Glimm} into a single statement, we obtain the following translations of the properties under study.

\begin{thm}\label{thm:EqChar}
    Let $A$ be a \ca{}. Then,
    \begin{enumerate}
        \item $A$ is nowhere scattered if and only if $\Cu (A)$ is weakly $(2,\omega )$-divisible\\
        (if and only if it is weakly $(k,\omega)$-divisible for every $k\in\NN$).
        \item $A$ has the Global Glimm Property if and only if $\Cu (A)$ is $(2,\omega )$-divisible\\
        (if and only if it is $(k,\omega)$-divisible for every $k\in\NN$).\footnote{That $(2,\omega)$-divisibility implies the Global Glimm Property was already contained in \cite[Theorem~5.3]{RobRor13Divisibility}.}
    \end{enumerate}
\end{thm}

\begin{cor}\label{prp:ConRobTik}
    Let $A$ be a separable \ca{} of finite nuclear dimension. Then,
    \begin{enumerate}
        \item $A$ is nowhere scattered if and only if $F(A)$ is.
        \item $F(A)$ has the Global Glimm Property if and only if $A$ is $\mathcal{Z}$-stable.
    \end{enumerate}
\end{cor}
\begin{proof}[Sketch of the proof]
    \autoref{prp:RobTikGGP} shows the forward implication of both (1) and (2). Thus, assume first that $F(A)$ is nowhere scattered. Then, it follows from \autoref{thm:EqChar} above that the unit in $F(A)$ is weakly $(2,\omega )$-divisible, that is, there exist $e_1,\ldots,e_n\in F(A)_+$ and $n\in\NN$ such that $e_i^{\oplus 2}\precsim 1_{F(A)}$ and $1_{F(A)}\precsim e_1\oplus\ldots\oplus e_n$ in the appropiate matrix amplifications of $F(A)$.

    Let $a\in A_+$. Embedding $A$ in $A_\mathcal{U}$ and `multiplying' the previous subequivalences by a suitable diagonal of $a$'s (using that $a$ commutes with every element in $F(A)$), we obtain
    \[
        (e_ia)^{\oplus 2}\precsim a,\andSep
        a\precsim e_1a\oplus\ldots\oplus e_na.
    \]

    This shows that $[a]$ is weakly $(2,\omega)$-divisible in $\Cu (A_\mathcal{U})$. Standard techniques can now be used to prove that, in fact, $[a]\in\Cu (A)$ is also weakly $(2,\omega)$-divisible in $\Cu  (A)$.\footnote{Crucially, in this step we use that the definition of weakly $(2,\omega)$-divisibility gives enough space (in the form of $x'\ll x$) to allow for cut-downs and other approximations.} Finally, \cite[Lemma 7.2]{ThiVil23arX:Soft} proves that $\Cu (A)$ is weakly $(2,\omega)$-divisible whenever $[a]\in\Cu (A)$ is weakly $(2,\omega)$-divisible for every $a\in A_+$.

    To prove (2), one can use that $A$ is $\mathcal{Z}$-stable if and only if $\mathcal{Z}$ embeds unitally in $F(A)$ (by \cite[Theorem~2.3]{TomWin07ssa} and \cite{Ror04StableRealRankZ}). Thus, the unit in $F(A)$ is almost divisible, and in particular $(2,\omega)$-divisible. Setting $e_1=\ldots=e_n$ in the proof of (1) shows that every element in $A_+$ is $(2,\omega )$-divisible. This is equivalent to the Global Glimm Property, as desired.\footnote{Alternatively, one can use the unital embedding $\mathcal{Z}\to F(A)$ and the existence of a full square-zero element in $\mathcal{Z}$ to obtain a full square-zero element $r$ in $F(A)$. Given any element $a\in F(A)_+$, there is a unital injective $^*$-homomorphism $\varphi$ from $C^*(1,r)$ to $F(A)\cap C^*(a)'$ by \cite[Proposition~1.12]{Kir06CentralSeqPI}. Thus, the element $a\varphi(r)a$ is a full square-zero element in $\overline{aF(A)a}$, which proves that $F(A)$ satisfies the Global Glimm Property.}
\end{proof}

With our new Cuntz semigroup picture at hand, we can translate the Global Glimm Problem (\autoref{qst:GGP}) as follows:
\begin{qst}\label{qst:AlgGGP}
    Let $A$ be a \ca{} whose Cuntz semigroup $\Cu (A)$ is weakly $(2,\omega )$-divisible. Is $\Cu (A)$ $(2,\omega )$-divisible?
\end{qst}

Note that this rephrasing of the problem now makes it a dimension reduction question: Assuming that you can `divide' each element using finitely many divisors (weak $(2,\omega)$-divisibility), can you make do with a single divisor ($(2,\omega)$-divisibility)?

\begin{rmk}\label{rmk:TWisDR}
    We stated in the introduction that (Q1) ---that is, the question of whether every separable nowhere scattered \ca{} of finite nuclear dimension is $\mathcal{Z}$-stable--- is a question about dimension reduction. What comes next is a justification of this statement:

    Let $A$ be a unital \ca{}, and let $s(1,A)$ denote the least natural number $k$ for which there exist nonzero pairwise distinct elements $y_1,\ldots ,y_k\in\Cu (A)$ such that $2y_i\leq [1]$ and $[1]\leq m_1y_1+\ldots +m_ky_k$ for some $m_1,\ldots ,m_k\in\NN\setminus\{ 0\}$. If no such $k$ exists, set $s(1,A)=\infty$.
    
    Note that $s(1,A)$ is finite whenever $\Cu (A)$ is weakly $(2,\omega)$-divisible (applying \autoref{dfn:Div} to the pair $[1]=[(1-\varepsilon)_+]\ll [1]$). Further, $s(1,A)=1$ whenever $\Cu (A)$ is $(2,\omega )$-divisible.

    Now let $A$ be a separable (possibly non-unital) \ca{}. Then, $F(A)$ is unital. One can show, following the same ideas from the proof of \autoref{prp:ConRobTik}, that $F(A)$ is nowhere scattered if and only if $s(1,F(A))$ is finite, while $F(A)$ has the Global Glimm Property if and only if $s(1,F(A))=1$.

    Using our reformulation of the Global Glimm Problem (\autoref{qst:AlgGGP}) and \autoref{cor:EqTWGGP}, we see that the following statements are equivalent:
    \begin{itemize}
        \item[(i)] Every separable nowhere scattered \ca{} of finite nuclear dimension is $\mathcal{Z}$-stable.
        \item[(ii)] For every separable nowhere scattered \ca{} $A$ of finite nuclear dimension we have that $s(1,F(A))$ is finite if and only if $s(1,F(A))=1$. 
    \end{itemize}

    Thus, (Q1) admits a rephrasing as a dimension reduction question.
\end{rmk}

Although \autoref{qst:AlgGGP} remains largely an analytic problem, the statement now has an algebraic flavor. Using this Cuntz semigroup translation, one can pinpoint the precise properties needed to upgrade nowhere scatteredness to the Global Glimm Property. These should be viewed as very weak versions of having infima and suprema in the Cuntz semigroup (\autoref{rmk:InfandSup}).

As with elements in a \ca{}, we use $[a]\lhd [b]$ to denote that $a$ is in the closed ideal generated by $b$. Further, we write $[a]\llideal [b]$ if there exists $\varepsilon>0$ such that $a\lhd (b-\varepsilon)_+$.

\begin{dfn}[{\cite[Section~4~\&~5]{ThiVil23Glimm}}]\label{dfn:IFandV}
    Let $A$ be a \ca{}. We say that its Cuntz semigroup $\Cu (A)$
    \begin{itemize}
        \item is \emph{ideal-filtered} if, for every $v',v,x,y\in\Cu (A)$ such that 
        \[
            v'\ll v\llideal x, y,
        \]
        there exists $z\in\Cu (A)$ such that 
        \[
            v'\llideal z \leq x,y.
        \]
        \item has \emph{property (V)}\footnote{Although this is not the original definition from \cite[Definition~5.1]{ThiVil23Glimm}, one can check that they are equivalent. We provide a sketch in this footnote in the separable case: First, note that \cite[Definition~5.1]{ThiVil23Glimm} trivially implies our notion. Conversely, if $x,d_1',d_1,d_2',d_2,c$ are elements in $\Cu (A)$ such that $d_i'\ll d_i\ll c$ and $c+d_i\ll x$ for $i=1,2$, let $x'\in\Cu (A)$ be $\ll $-below $x\wedge \infty(d_1+d_2)$ (the existence of this infimum is guaranteed by \cite[Theorem~2.4]{AntPerRobThi21Edwards}) but large enough so that $d_i'\ll x'$. Using property (V) as stated in this survey, one checks that the elements $y,z$ satisfy the desired conditions in \cite[Definition~5.1]{ThiVil23Glimm}.} if, for every $c,d_1,d_2,x',x\in\Cu (A)$ such that
        \[
            x'\ll x,\quad 
             d_1,d_2\ll c,\quad 
            c+d_1,c+d_2\ll x,\andSep  
            x'\llideal d_1+d_2
        \]
        there exist $y,z\in\Cu (A)$ such that
        \[
            y+z\leq x,\andSep
            x'\llideal y, z.
        \]
    \end{itemize}
\end{dfn}

\begin{rmk}\label{rmk:InfandSup}
    Let us assume that $\Cu (A)$ is \emph{inf-semilattice ordered}, that is, infima of finite sets exist and the equality $(x+z)\wedge (y+z)=(x\wedge y)+z$ holds. If $\Cu (A)$ is inf-semilattice ordered, then it is ideal-filtered. Indeed, simply set $z=x\wedge y$ and note that $v'\ll v\lhd z$. As an example, the Cuntz semigroup of any separable \ca{} of stable rank one is inf-semilattice ordered and thus ideal-filtered; see \cite[Theorem~3.8]{AntPerRobThi22CuntzSR1}.

    Now assume that $\Cu (A)$ is \emph{sup-semilattice ordered}, that is, suprema of finite sets exist and the equality $(x+z)\vee (y+z)=(x\vee y)+z$ holds. Then, one sees that $\Cu (A)$ has property (V) by setting $y=c$ and $z=d_1\vee d_2$. An example of sup-semilattice ordered Cuntz semigroup is given by any commutative \ca{} with spectrum of dimension at most $1$; see e.g. \cite[Lemma~4.17]{Vil21arX:CommCuAI}.

    It is of note that being inf-semilattice ordered happens often (perhaps even for the Cuntz semigroup of every pure \ca{}), while being sup-semilattice ordered is much rarer. This imbalance between the notions can be clearly seen in the ordered semigroup ${\rm LAff}(K,(-\infty ,\infty])$ of lower-semicontinuous, affine functions on a Choquet simplex $K$, which is always an inf-semilattice \cite[Lemma~3.7]{Thi20RksOps} but is a sup-semilattice if and only if $K$ is Bauer; cf. \cite[Theorem~II.4.1]{Alf71CpctCvxSets}.
\end{rmk}

The importance of ideal-filteredness and property (V) is that they are satisfied in a variety of situations regardless of nowhere scatteredness. We note that, while in \cite{ThiVil23Glimm} it was unclear how to deduce property (V) from topological dimension zero, we show in \autoref{thm:IFandVList} below that the reformulation of property (V) given here does indeed follow from this assumption.

\begin{thm}\label{thm:IFandVList}
    Let $A$ be a \ca{} of either topological dimension zero or stable rank one. Then, $\Cu (A)$ is ideal-filtered and has property (V).
\end{thm}
\begin{proof}[Idea of the proof and references]
    \cite[Section~4]{ThiVil23Glimm} contains the proofs that either topological dimension zero or stable rank one implies ideal-filteredness of $\Cu (A)$.     \cite[Theorem~5.6]{ThiVil23Glimm} shows that $\Cu (A)$ has property (V) whenever $A$ is residually stably finite. In particular, $\Cu (A)$ has property (V) whenever $A$ is of stable rank one.

    Finally, assume that $A$ has topological dimension zero. It is shown in the proof of \cite[Theorem~2.3]{NgThiVil25arX} that, whenever $d_1,d_2,x',x\in\Cu (A)$ satisfy 
    \[
        x'\ll x,\quad 
        2d_1, 2d_2\ll x,\andSep 
        x'\llideal d_1+d_2,
    \]
    then there exists $d\in\Cu (A)$ such that $2d\leq x$ and $x'\llideal d$.

    This property implies (V). Indeed, if $c,d_1,d_2,x',x\in\Cu (A)$ are as in \autoref{dfn:IFandV}, we get $2d_i\leq c+d_i\ll x$ for $i=1,2$. Thus, we can apply the stated property to obtain an element $d$ such that $2d\leq x$ and $x'\llideal  d$. Setting $y=z=d$, we see that $\Cu (A)$ has (V).
\end{proof}

As mentioned, one can use \autoref{thm:EqChar} to show that the Global Glimm Property implies the two notions from \autoref{dfn:IFandV}. That the converse holds under the presence of nowhere scatteredness is the main result of \cite{ThiVil23Glimm}. We give here some information about the proof, since it highlights the different nature of each property: While ideal-filteredness plays an instrumental role, property (V) appears to be a technical ---and perpahs automatic--- tool; see \autoref{qst:IFandV} below.

\begin{thm}[{\cite[Theorem~7.1]{ThiVil23Glimm}}]\label{thm:MainGGP}
    Let $A$ be a \ca{}. Then, the following are equivalent:
    \begin{itemize}
        \item[(i)] $A$ is nowhere scattered and $\Cu (A)$ is ideal-filtered and has property (V).
        \item[(ii)] $A$ has the Global Glimm Property.
    \end{itemize}
\end{thm}
\begin{proof}[Sketch of the proof]
    Let $A$ be nowhere scattered, and let us prove that $\Cu (A)$ is $(2,\omega )$-divisible. First, we note that this condition is equivalent to the following property: For any $x',x\in\Cu (A)$ such that $x'\ll x$, there exists $y\in\Cu (A)$ such that $2d\leq x$ and $x'\llideal d$.
    
    Thus, let $x',x\in\Cu (A)$ be such that $x'\ll x$ and let us find such a $d$. Take $x''\in\Cu (A)$ such that $x'\ll x''\ll x$. As shown in \cite[Proposition~6.2]{ThiVil23Glimm}, ideal-filteredness alone is enough to reduce the number of divisors to two, that is, there exist $d_1,d_2,c\in\Cu (A)$ such that
    \[
            d_1,d_2\ll c,\quad 
            c+d_1,c+d_2\ll x,\andSep 
            x''\llideal d_1+d_2.
    \]
    In particular, we have $2d_1,2d_2\leq x$ and $x''\llideal d_1+d_2$. Thus, what remains to show is that property (V) can be used to `combine' $d_1$ and $d_2$ into a single divisor $d$.

    First, using (V) for $d_1,d_2,c,x'',x$, we obtain $y,z\in\Cu (A)$ such that $y+z\leq x$, $x''\llideal y$ and $x''\llideal z$. Now, applying ideal-filteredness again to 
    \[
        x'\ll x''\llideal y, z,
    \]
    we obtain $d\in\Cu (A)$ such that $x'\llideal d$ and $d\leq y,z$.

    Thus, we get $2d\leq y+z\leq x$, as required.
\end{proof}

With the previous result at hand, \autoref{qst:AlgGGP} can be rephrased once more by asking if every nowhere scattered \ca{} has an ideal-filtered Cuntz semigroup with property (V). We are not aware of any \ca{} whose Cuntz semigroup does not satisfy property (V)\footnote{In contrast, one can show that the Cuntz semigroup of $C(S^2)$ is not ideal-filtered \cite[Example~4.3]{ThiVil23Glimm}.}. Thus, we ask:

\begin{qst}\label{qst:IFandV}
    Let $A$ be a \ca{}. Does $\Cu (A)$ have property (V)? If $A$ is nowhere scattered, is $\Cu (A)$ ideal-filtered?
\end{qst}

Combining our results thus far we can now prove Theorems \ref{thm:GGPdim0} and \ref{thm:GGPsr1}.

\begin{cor}\label{cor:GGPsr1dim0}
    Let $A$ be a \ca{} of either topological dimension zero or stable rank one. Then, $A$ is nowhere scattered if and only if it has the Global Glimm Property.
\end{cor}
\begin{proof}
    \autoref{thm:IFandVList} shows that $A$ has an ideal-filtered Cuntz semigroup with property (V). \autoref{thm:MainGGP} gives the desired equivalence.
\end{proof}

\stoptoc

\providecommand{\etalchar}[1]{$^{#1}$}
\providecommand{\bysame}{\leavevmode\hbox to3em{\hrulefill}\thinspace}
\providecommand{\noopsort}[1]{}
\providecommand{\mr}[1]{\href{http://www.ams.org/mathscinet-getitem?mr=#1}{MR~#1}}
\providecommand{\zbl}[1]{\href{http://www.zentralblatt-math.org/zmath/en/search/?q=an:#1}{Zbl~#1}}
\providecommand{\jfm}[1]{\href{http://www.emis.de/cgi-bin/JFM-item?#1}{JFM~#1}}
\providecommand{\arxiv}[1]{\href{http://www.arxiv.org/abs/#1}{arXiv~#1}}
\providecommand{\doi}[1]{\url{http://dx.doi.org/#1}}
\providecommand{\MR}{\relax\ifhmode\unskip\space\fi MR }
\providecommand{\MRhref}[2]{%
  \href{http://www.ams.org/mathscinet-getitem?mr=#1}{#2}
}
\providecommand{\href}[2]{#2}


\begin{thebibliography}{AGKEP25}

\bibitem[Alf71]{Alf71CpctCvxSets}
\bgroup\scshape{}E.~M. Alfsen\egroup{}, \emph{Compact convex sets and boundary
  integrals}, Springer-Verlag, New York-Heidelberg, 1971, Ergebnisse der
  Mathematik und ihrer Grenzgebiete, Band 57.

\bibitem[AGKEP25]{AGKEP24:StCompTwisted}
\bgroup\scshape{}T.~Amrutam\egroup{}, \bgroup\scshape{}D.~Gao\egroup{},
  \bgroup\scshape{}S.~Kunnawalkam~Elayavalli\egroup{}, and
  \bgroup\scshape{}G.~Patchell\egroup{}, Strict comparison in reduced group
  \ca{s}, Invent. Math. (to appear), preprint (arXiv:2412.06031 [math.OA]),
  2025.

\bibitem[APRT21]{AntPerRobThi21Edwards}
\bgroup\scshape{}R.~Antoine\egroup{}, \bgroup\scshape{}F.~Perera\egroup{},
  \bgroup\scshape{}L.~Robert\egroup{}, and \bgroup\scshape{}H.~Thiel\egroup{},
  Edwards' condition for quasitraces on \ca{s},  \emph{Proc. Roy. Soc.
  Edinburgh Sect. A} \textbf{151} (2021), 525--547.

\bibitem[APRT22]{AntPerRobThi22CuntzSR1}
\bgroup\scshape{}R.~Antoine\egroup{}, \bgroup\scshape{}F.~Perera\egroup{},
  \bgroup\scshape{}L.~Robert\egroup{}, and \bgroup\scshape{}H.~Thiel\egroup{},
  \ca{s} of stable rank one and their {C}untz semigroups,  \emph{Duke Math. J.}
  \textbf{171} (2022), 33--99.

\bibitem[APRT24]{AntPerRobThi24TracesUltra}
\bgroup\scshape{}R.~Antoine\egroup{}, \bgroup\scshape{}F.~Perera\egroup{},
  \bgroup\scshape{}L.~Robert\egroup{}, and \bgroup\scshape{}H.~Thiel\egroup{},
  Traces on ultrapowers of \ca{s},  \emph{J. Funct. Anal.} \textbf{286} (2024),
  65 pages.

\bibitem[APT18]{AntPerThi18TensorProdCu}
\bgroup\scshape{}R.~Antoine\egroup{}, \bgroup\scshape{}F.~Perera\egroup{}, and
  \bgroup\scshape{}H.~Thiel\egroup{}, Tensor products and regularity properties
  of {C}untz semigroups,  \emph{Mem. Amer. Math. Soc.} \textbf{251} (2018),
  viii+191.

\bibitem[APTV24]{AntPerThiVil24arX:PureCAlgs}
\bgroup\scshape{}R.~Antoine\egroup{}, \bgroup\scshape{}F.~Perera\egroup{},
  \bgroup\scshape{}H.~Thiel\egroup{}, and \bgroup\scshape{}E.~Vilalta\egroup{},
  Pure \ca{s}, preprint (arXiv:2406.11052 [math.OA]), 2024.

\bibitem[AVTV25]{AsaThiVil23arX:RksSoftOps}
\bgroup\scshape{}M.~A. Asadi-Vasfi\egroup{},
  \bgroup\scshape{}H.~Thiel\egroup{}, and \bgroup\scshape{}E.~Vilalta\egroup{},
  Ranks of soft operators in nowhere scattered \ca{s},  \emph{J. Inst. Math.
  Jussieu} \textbf{24} (2025), 371--410.

\bibitem[BEvV22]{BedEnsvVelt22}
\bgroup\scshape{}E.~B\'edos\egroup{}, \bgroup\scshape{}U.~Enstad\egroup{}, and
  \bgroup\scshape{}J.~T. van Velthoven\egroup{}, Smooth lattice orbits of
  nilpotent groups and strict comparison of projections,  \emph{J. Funct.
  Anal.} \textbf{283} (2022), Paper No. 109572, 48.

\bibitem[Bla06]{Bla06OpAlgs}
\bgroup\scshape{}B.~Blackadar\egroup{}, \emph{Operator algebras},
  \emph{Encyclopaedia of Mathematical Sciences} \textbf{122}, Springer-Verlag,
  Berlin, 2006, Theory of \ca{s} and von Neumann algebras, Operator Algebras
  and Non-commutative Geometry, III.

\bibitem[BH82]{BlaHan82DimFct}
\bgroup\scshape{}B.~Blackadar\egroup{} and
  \bgroup\scshape{}D.~Handelman\egroup{}, Dimension functions and traces on
  \ca{s},  \emph{J. Funct. Anal.} \textbf{45} (1982), 297--340.

\bibitem[BKR92]{BlaKumRor92ApproxCentralMatUnit}
\bgroup\scshape{}B.~Blackadar\egroup{}, \bgroup\scshape{}A.~Kumjian\egroup{},
  and \bgroup\scshape{}M.~R{\o}rdam\egroup{}, Approximately central matrix
  units and the structure of noncommutative tori,  \emph{$K$-Theory} \textbf{6}
  (1992), 267--284.

\bibitem[BK04a]{BlaKir04GlimmHalving}
\bgroup\scshape{}E.~Blanchard\egroup{} and
  \bgroup\scshape{}E.~Kirchberg\egroup{}, Global {G}limm halving for
  {$C^*$}-bundles,  \emph{J. Operator Theory} \textbf{52} (2004), 385--420.

\bibitem[BK04b]{BlaKir04PureInf}
\bgroup\scshape{}E.~Blanchard\egroup{} and
  \bgroup\scshape{}E.~Kirchberg\egroup{}, Non-simple purely infinite \ca{s}:
  the {H}ausdorff case,  \emph{J. Funct. Anal.} \textbf{207} (2004), 461--513.

\bibitem[BP91]{BroPed91CAlgRR0}
\bgroup\scshape{}L.~G. Brown\egroup{} and \bgroup\scshape{}G.~K.
  Pedersen\egroup{}, \ca{s} of real rank zero,  \emph{J. Funct. Anal.}
  \textbf{99} (1991), 131--149.

\bibitem[CETW22]{CasEviTikWhi22UniPropGamma}
\bgroup\scshape{}J.~Castillejos\egroup{},
  \bgroup\scshape{}S.~Evington\egroup{}, \bgroup\scshape{}A.~Tikuisis\egroup{},
  and \bgroup\scshape{}S.~White\egroup{}, Uniform property {$\Gamma$},
  \emph{Int. Math. Res. Not. IMRN} (2022), 9864--9908.

\bibitem[CET{\etalchar{+}}21]{CasEviTikWhiWin21NucDimSimple}
\bgroup\scshape{}J.~Castillejos\egroup{},
  \bgroup\scshape{}S.~Evington\egroup{}, \bgroup\scshape{}A.~Tikuisis\egroup{},
  \bgroup\scshape{}S.~White\egroup{}, and \bgroup\scshape{}W.~Winter\egroup{},
  Nuclear dimension of simple \ca{s},  \emph{Invent. Math.} \textbf{224}
  (2021), 245--290.

\bibitem[CEI08]{CowEllIva08CuInv}
\bgroup\scshape{}K.~T. Coward\egroup{}, \bgroup\scshape{}G.~A.
  Elliott\egroup{}, and \bgroup\scshape{}C.~Ivanescu\egroup{}, The {C}untz
  semigroup as an invariant for \ca{s},  \emph{J.\ Reine Angew.\ Math.}
  \textbf{623} (2008), 161--193.

\bibitem[Cun78]{Cun78DimFct}
\bgroup\scshape{}J.~Cuntz\egroup{}, Dimension functions on simple \ca{s},
  \emph{Math. Ann.} \textbf{233} (1978), 145--153.

\bibitem[Cun81]{Cun81Kth}
\bgroup\scshape{}J.~Cuntz\egroup{}, {$K$}-theory for certain {$C\sp{\ast}
  $}-algebras,  \emph{Ann. of Math. (2)} \textbf{113} (1981), 181--197.

\bibitem[DT10]{DadTom10Ranks}
\bgroup\scshape{}M.~Dadarlat\egroup{} and \bgroup\scshape{}A.~S. Toms\egroup{},
  Ranks of operators in simple \ca{s},  \emph{J. Funct. Anal.} \textbf{259}
  (2010), 1209--1229.

\bibitem[EM18]{EckMcK18DimNucNilpotGps}
\bgroup\scshape{}C.~Eckhardt\egroup{} and
  \bgroup\scshape{}P.~McKenney\egroup{}, Finitely generated nilpotent group
  \ca{s} have finite nuclear dimension,  \emph{J. Reine Angew. Math.}
  \textbf{738} (2018), 281--298.

\bibitem[ERS11]{EllRobSan11Cone}
\bgroup\scshape{}G.~A. Elliott\egroup{}, \bgroup\scshape{}L.~Robert\egroup{},
  and \bgroup\scshape{}L.~Santiago\egroup{}, The cone of lower semicontinuous
  traces on a \ca{},  \emph{Amer. J. Math.} \textbf{133} (2011), 969--1005.

\bibitem[ER06]{EllRor06Perturb}
\bgroup\scshape{}G.~A. Elliott\egroup{} and
  \bgroup\scshape{}M.~R{\o}rdam\egroup{}, Perturbation of {H}ausdorff moment
  sequences, and an application to the theory of \ca{s} of real rank zero,  in
  \emph{Operator {A}lgebras: {T}he {A}bel {S}ymposium 2004}, \emph{Abel Symp.}
  \textbf{1}, Springer, Berlin, 2006, pp.~97--115.

\bibitem[ETV25]{EnsThiVil25Criteria}
\bgroup\scshape{}U.~Enstad\egroup{}, \bgroup\scshape{}H.~Thiel\egroup{}, and
  \bgroup\scshape{}E.~Vilalta\egroup{}, Criteria for the existence of
  {S}chwartz {G}abor frames over rational lattices,  \emph{Int. Math. Res. Not.
  IMRN} (2025), Paper No. rnaf038, 12.

\bibitem[EV25]{EnsVil25arX:Twisted}
\bgroup\scshape{}U.~Enstad\egroup{} and \bgroup\scshape{}E.~Vilalta\egroup{},
  $\mathcal{Z}$-stability of twisted group \ca{s} of nilpotent groups, preprint
  (arXiv:2503.18088 [math.OA]), 2025.

\bibitem[Fau25]{Fau25ZstabGraAlg}
\bgroup\scshape{}G.~Faurot\egroup{}, $\mathcal{Z}$-stable graph algebras,
  preprint (arXiv:2511.02760 [math.OA]), 2025.

\bibitem[GS24]{GabSza24Dyn}
\bgroup\scshape{}J.~Gabe\egroup{} and \bgroup\scshape{}G.~Szab\'o\egroup{}, The
  dynamical {K}irchberg-{P}hillips theorem,  \emph{Acta Math.} \textbf{232}
  (2024), 1--77.

\bibitem[GP23]{GarPer23arX:ModernCu}
\bgroup\scshape{}E.~Gardella\egroup{} and \bgroup\scshape{}F.~Perera\egroup{},
  The modern theory of {C}untz semigroups of \ca{s}, EMS Surveys (to appear),
  preprint (arXiv:2212.02290 [math.OA]), 2023.

\bibitem[GK18]{GhaKos18NCCantorBendixson}
\bgroup\scshape{}S.~Ghasemi\egroup{} and
  \bgroup\scshape{}P.~Koszmider\egroup{}, Noncommutative {C}antor-{B}endixson
  derivatives and scattered \ca{s},  \emph{Topology Appl.} \textbf{240} (2018),
  183--209.

\bibitem[Gli61]{Gli61Type1}
\bgroup\scshape{}J.~Glimm\egroup{}, Type {I} \ca{s},  \emph{Ann. of Math. (2)}
  \textbf{73} (1961), 572--612.

\bibitem[Gon02]{Gon02}
\bgroup\scshape{}G.~Gong\egroup{}, On the classification of simple inductive
  limit \ca{s}. {I}: The reduction theorem,  \emph{Doc. Math.} \textbf{7}
  (2002), 255–461.

\bibitem[GLN20]{GonLinNiu20}
\bgroup\scshape{}G.~Gong\egroup{}, \bgroup\scshape{}H.~Lin\egroup{}, and
  \bgroup\scshape{}Z.~Niu\egroup{}, A classification of finite simple amenable
  $\mathcal{Z}$-stable \ca{s}, {I}{I}: \ca{s} with rational generalized tracial
  rank one,  \emph{C. R. Math. Acad. Sci. Soc. R. Can.} \textbf{42} (2020),
  451–539.

\bibitem[Haa14]{Haa14Quasitraces}
\bgroup\scshape{}U.~Haagerup\egroup{}, Quasitraces on exact \ca{s} are traces,
  \emph{C. R. Math. Acad. Sci. Soc. R. Can.} \textbf{36} (2014), 67--92.

\bibitem[Jen77]{Jen77ScatteredCAlg}
\bgroup\scshape{}H.~E.~n. Jensen\egroup{}, Scattered \ca{s},  \emph{Math.
  Scand.} \textbf{41} (1977), 308--314.

\bibitem[Jen78]{Jen78ScatteredCAlg2}
\bgroup\scshape{}H.~E.~n. Jensen\egroup{}, Scattered \ca{s}. {II},  \emph{Math.
  Scand.} \textbf{43} (1978), 308--310 (1979).

\bibitem[JS99]{JiaSu99Projectionless}
\bgroup\scshape{}X.~Jiang\egroup{} and \bgroup\scshape{}H.~Su\egroup{}, On a
  simple unital projectionless \ca{},  \emph{Amer. J. Math.} \textbf{121}
  (1999), 359--413.

\bibitem[Kir94]{Kir94ICM}
\bgroup\scshape{}E.~Kirchberg\egroup{}, Exact \ca{s}, tensor products, and the
  classification of purely infinite algebras,  in \emph{Proceedings of the
  International Congress of Mathematicians, Vol. 1, 2}, Birkh\"{a}user, Basel,
  1994, pp.~943--–954.

\bibitem[Kir06]{Kir06CentralSeqPI}
\bgroup\scshape{}E.~Kirchberg\egroup{}, Central sequences in \ca{s} and
  strongly purely infinite algebras,  in \emph{Operator {A}lgebras: {T}he
  {A}bel {S}ymposium 2004}, \emph{Abel Symp.} \textbf{1}, Springer, Berlin,
  2006, pp.~175--231.

\bibitem[Kir25]{Kir25Book}
\bgroup\scshape{}E.~Kirchberg\egroup{}, \emph{The Classification of Purely
  Infinite C*-Algebras Using Kasparov’s Theory}, electronic ed. ed.,
  University Library of the University of Münster, 2025.

\bibitem[KR00]{KirRor00PureInf}
\bgroup\scshape{}E.~Kirchberg\egroup{} and
  \bgroup\scshape{}M.~R{\o}rdam\egroup{}, Non-simple purely infinite \ca{s},
  \emph{Amer. J. Math.} \textbf{122} (2000), 637--666.

\bibitem[KR02]{KirRor02InfNonSimpleCalgAbsOInfty}
\bgroup\scshape{}E.~Kirchberg\egroup{} and
  \bgroup\scshape{}M.~R{\o}rdam\egroup{}, Infinite non-simple \ca{s}: absorbing
  the {C}untz algebra {$\mathcal{O}_\infty$},  \emph{Adv. Math.} \textbf{167}
  (2002), 195--264.

\bibitem[KR14]{KirRor14CentralSeq}
\bgroup\scshape{}E.~Kirchberg\egroup{} and
  \bgroup\scshape{}M.~R{\o}rdam\egroup{}, Central sequence \ca{s} and tensorial
  absorption of the {J}iang-{S}u algebra,  \emph{J. Reine Angew. Math.}
  \textbf{695} (2014), 175--214.

\bibitem[KR15]{KirRor15CentralSeqCharacters}
\bgroup\scshape{}E.~Kirchberg\egroup{} and
  \bgroup\scshape{}M.~R{\o}rdam\egroup{}, When central sequence \ca{s} have
  characters,  \emph{Internat. J. Math.} \textbf{26} (2015), 1550049, 32.

\bibitem[KPRR97]{KPRR97}
\bgroup\scshape{}A.~Kumjian\egroup{}, \bgroup\scshape{}D.~Pask\egroup{},
  \bgroup\scshape{}I.~Raeburn\egroup{}, and
  \bgroup\scshape{}J.~Renault\egroup{}, Graphs, groupoids, and
  {C}untz–{K}rieger algebras,  \emph{J. Funct. Anal.} \textbf{144} (1997),
  505--541.

\bibitem[Lin25]{Lin25Sr1}
\bgroup\scshape{}H.~Lin\egroup{}, Strict comparison and stable rank one,
  \emph{J. Funct. Anal.} \textbf{289} (2025), Paper No. 111065, 25.

\bibitem[MS12]{MatSat12StrComparison}
\bgroup\scshape{}H.~Matui\egroup{} and \bgroup\scshape{}Y.~Sato\egroup{},
  Strict comparison and {$\mathcal{Z}$}-absorption of nuclear \ca{s},
  \emph{Acta Math.} \textbf{209} (2012), 179--196.

\bibitem[NTV25]{NgThiVil25arX}
\bgroup\scshape{}P.~W. Ng\egroup{}, \bgroup\scshape{}H.~Thiel\egroup{}, and
  \bgroup\scshape{}E.~Vilalta\egroup{}, The {G}lobal {G}limm {P}roperty for
  \ca{s} of topological dimension zero, preprint (arXiv:2507.16261 [math.OA]),
  2025.

\bibitem[Oza25]{Oza25Arx}
\bgroup\scshape{}N.~Ozawa\egroup{}, Proximality and selflessness for group
  \ca{s}, preprint (arXiv:2508.07938 [math.OA]), 2025.

\bibitem[Phi00]{Phi00ClassNuclPISimple}
\bgroup\scshape{}N.~C. Phillips\egroup{}, A classification theorem for nuclear
  purely infinite simple \ca{s},  \emph{Doc. Math.} \textbf{5} (2000), 49--114.

\bibitem[Rie83]{Rie83DimSRKThy}
\bgroup\scshape{}M.~A. Rieffel\egroup{}, Dimension and stable rank in the
  {$K$}-theory of \ca{s},  \emph{Proc. London Math. Soc. (3)} \textbf{46}
  (1983), 301--333.

\bibitem[Rob25]{Rob25Self}
\bgroup\scshape{}L.~Robert\egroup{}, Selfless \ca{s},  \emph{Adv. Math.}
  \textbf{478} (2025), paper No. 110409.

\bibitem[RR13]{RobRor13Divisibility}
\bgroup\scshape{}L.~Robert\egroup{} and \bgroup\scshape{}M.~R{\o}rdam\egroup{},
  Divisibility properties for \ca{s},  \emph{Proc. Lond. Math. Soc. (3)}
  \textbf{106} (2013), 1330--1370.

\bibitem[RT11]{RobTik11HilbCommutative}
\bgroup\scshape{}L.~Robert\egroup{} and \bgroup\scshape{}A.~Tikuisis\egroup{},
  Hilbert {$C^*$}-modules over a commutative \ca{},  \emph{Proc. Lond. Math.
  Soc. (3)} \textbf{102} (2011), 229--256.

\bibitem[RT17]{RobTik17NucDimNonSimple}
\bgroup\scshape{}L.~Robert\egroup{} and \bgroup\scshape{}A.~Tikuisis\egroup{},
  Nuclear dimension and {$\mathcal{Z}$}-stability of non-simple \ca{s},
  \emph{Trans. Amer. Math. Soc.} \textbf{369} (2017), 4631--4670.

\bibitem[R{\o}r04]{Ror04StableRealRankZ}
\bgroup\scshape{}M.~R{\o}rdam\egroup{}, The stable and the real rank of
  {$\mathcal{Z}$}-absorbing \ca{s},  \emph{Internat. J. Math.} \textbf{15}
  (2004), 1065--1084.

\bibitem[R\o06]{Ror07ICM}
\bgroup\scshape{}M.~R\o{}rdam\egroup{}, Structure and classification of \ca{s},
   in \emph{{I}nternational {C}ongress of {M}athematicians, {V}ol. {I}}, Eur.
  Math. Soc., Zurich, 2006, pp.~1581--1598.

\bibitem[RW10]{RorWin10ZRevisited}
\bgroup\scshape{}M.~R{\o}rdam\egroup{} and \bgroup\scshape{}W.~Winter\egroup{},
  The {J}iang-{S}u algebra revisited,  \emph{J. Reine Angew. Math.}
  \textbf{642} (2010), 129--155.

\bibitem[Sat12]{Sat12arx:TraceSpace}
\bgroup\scshape{}Y.~Sato\egroup{}, Trace spaces of simple nuclear \ca{s} with
  finite-dimensional extreme boundary, preprint (arXiv:1209.3000 [math.OA]),
  2012.

\bibitem[STW25]{SchTikWhi99arX:Prob}
\bgroup\scshape{}C.~Schafhauser\egroup{},
  \bgroup\scshape{}A.~Tikuisis\egroup{}, and
  \bgroup\scshape{}S.~White\egroup{}, Nuclear \ca{s}: 99 problems, preprint
  (arXiv:2506.10902 [math.OA]), 2025.

\bibitem[Thi20]{Thi20RksOps}
\bgroup\scshape{}H.~Thiel\egroup{}, Ranks of operators in simple \ca{s} with
  stable rank one,  \emph{Comm. Math. Phys.} \textbf{377} (2020), 37--76.

\bibitem[Thi24]{Thi20arX:diffuseHaar}
\bgroup\scshape{}H.~Thiel\egroup{}, Diffuse traces and {H}aar unitaries,
  \emph{Amer. J. Math.} \textbf{146} (2024), 1305--1337.

\bibitem[TV23]{ThiVil23Glimm}
\bgroup\scshape{}H.~Thiel\egroup{} and \bgroup\scshape{}E.~Vilalta\egroup{},
  The {G}lobal {G}limm {P}roperty,  \emph{Trans. Amer. Math. Soc.} \textbf{376}
  (2023), 4713--4744.

\bibitem[TV24a]{ThiVil24NowhereScattered}
\bgroup\scshape{}H.~Thiel\egroup{} and \bgroup\scshape{}E.~Vilalta\egroup{},
  Nowhere scattered \ca{s},  \emph{J. Noncommut. Geom.} \textbf{18} (2024),
  231--263.

\bibitem[TV24b]{ThiVil23arX:Soft}
\bgroup\scshape{}H.~Thiel\egroup{} and \bgroup\scshape{}E.~Vilalta\egroup{},
  Soft operators in {${\rm C}^*$}-algebras,  \emph{J. Math. Anal. Appl.}
  \textbf{536} (2024), Paper No. 128167, 35.

\bibitem[TWW17]{TikWhiWin17QDNuclear}
\bgroup\scshape{}A.~Tikuisis\egroup{}, \bgroup\scshape{}S.~White\egroup{}, and
  \bgroup\scshape{}W.~Winter\egroup{}, Quasidiagonality of nuclear \ca{s},
  \emph{Ann. of Math. (2)} \textbf{185} (2017), 229--284.

\bibitem[TW14]{TikWin14DecompRank}
\bgroup\scshape{}A.~Tikuisis\egroup{} and \bgroup\scshape{}W.~Winter\egroup{},
  Decomposition rank of {$\mathcal{Z}$}-stable {$\rm C^*$}-algebras,
  \emph{Anal. PDE} \textbf{7} (2014), 673--700.

\bibitem[TWW15]{TomWhiWin15ZStableFdBauer}
\bgroup\scshape{}A.~S. Toms\egroup{}, \bgroup\scshape{}S.~White\egroup{}, and
  \bgroup\scshape{}W.~Winter\egroup{}, {$\mathcal{Z}$}-stability and
  finite-dimensional tracial boundaries,  \emph{Int. Math. Res. Not. IMRN}
  (2015), 2702--2727.

\bibitem[TW07]{TomWin07ssa}
\bgroup\scshape{}A.~S. Toms\egroup{} and \bgroup\scshape{}W.~Winter\egroup{},
  Strongly self-absorbing \ca{s},  \emph{Trans. Amer. Math. Soc.} \textbf{359}
  (2007), 3999--4029.

\bibitem[TW08]{TomWin08ASH}
\bgroup\scshape{}A.~S. Toms\egroup{} and \bgroup\scshape{}W.~Winter\egroup{},
  {$\mathcal{Z}$}-stable {ASH} algebras,  \emph{Canad. J. Math.} \textbf{60}
  (2008), 703--720.

\bibitem[TW09]{TomWin09Villadsen}
\bgroup\scshape{}A.~S. Toms\egroup{} and \bgroup\scshape{}W.~Winter\egroup{},
  The {E}lliott conjecture for {V}illadsen algebras of the first type,
  \emph{J. Funct. Anal.} \textbf{256} (2009), 1311--1340.

\bibitem[Vil23]{Vil21arX:CommCuAI}
\bgroup\scshape{}E.~Vilalta\egroup{}, The {C}untz semigroup of unital
  commutative {AI}-algebras,  \emph{Canad. J. Math.} \textbf{75} (2023),
  1831--1868.

\bibitem[Vil25]{Vil23NScaMult}
\bgroup\scshape{}E.~Vilalta\egroup{}, Nowhere scattered multiplier algebras,
  \emph{Proc. Roy. Soc. Edinburgh Sect. A} \textbf{155} (2025), 1113–1142.

\bibitem[Whi23]{Whi23ICM}
\bgroup\scshape{}S.~White\egroup{}, Abstract classification theorems for
  amenable {$C^\ast$}-algebras,  in \emph{I{CM}---{I}nternational {C}ongress of
  {M}athematicians. {V}ol. 4. {S}ections 5--8}, EMS Press, Berlin, 2023,
  pp.~3314--3338.

\bibitem[Win12]{Win12NuclDimZstable}
\bgroup\scshape{}W.~Winter\egroup{}, Nuclear dimension and
  {$\mathcal{Z}$}-stability of pure \ca{s},  \emph{Invent. Math.} \textbf{187}
  (2012), 259--342.

\bibitem[Win14]{Win14}
\bgroup\scshape{}W.~Winter\egroup{}, Localizing the elliott conjecture at
  strongly self-absorbing \ca{s},  \emph{J. Reine Angew. Math.} \textbf{692}
  (2014), 193–231.

\bibitem[Win18]{Win18ICM}
\bgroup\scshape{}W.~Winter\egroup{}, Structure of nuclear \ca{s}: from
  quasidiagonality to classification and back again,  in \emph{Proceedings of
  the {I}nternational {C}ongress of {M}athematicians---{R}io de {J}aneiro 2018.
  {V}ol. {III}. {I}nvited lectures}, World Sci. Publ., Hackensack, NJ, 2018,
  pp.~1801--1823.

\bibitem[WZ10]{WinZac10NuclDim}
\bgroup\scshape{}W.~Winter\egroup{} and \bgroup\scshape{}J.~Zacharias\egroup{},
  The nuclear dimension of \ca{s},  \emph{Adv. Math.} \textbf{224} (2010),
  461--498.

\end{thebibliography}
\end{document}